\documentclass{amsart}
\usepackage{amssymb}
\usepackage{amsmath}
\usepackage{latexsym}
\usepackage{eepic}
\usepackage{epsfig}
\usepackage{graphicx}
\usepackage{pb-diagram,lamsarrow,pb-lams,amscd}
\usepackage{eufrak}
\usepackage{mathrsfs}
\usepackage[english]{babel}
\usepackage[all,cmtip]{xy}
\usepackage{enumerate}

\DeclareMathAlphabet{\mathpzc}{OT1}{pzc}{m}{it}
\newtheorem{theorem}{Theorem}[subsection]
\newtheorem{theorem-definition}[theorem]{Theorem-Definition}
\newtheorem{lemma-definition}[theorem]{Lemma-Definition}
\newtheorem{definition-prop}[theorem]{Proposition-Definition}
\newtheorem{corollary}[theorem]{Corollary}
\newtheorem{prop}[theorem]{Proposition}
\newtheorem{lemma}[theorem]{Lemma}
\newtheorem{cor}[theorem]{Corollary}
\newtheorem{definition}[theorem]{Definition}

\newtheorem{question}[theorem]{Question}
\newtheorem{stelling?}{Stelling(?)}[section]

\theoremstyle{definition}
\newtheorem{remark}[theorem]{Remark}

\newcommand{\N}{\ensuremath{\mathbb{N}}}
\newcommand{\Z}{\ensuremath{\mathbb{Z}}}
\newcommand{\Q}{\ensuremath{\mathbb{Q}}}

\newcommand{\A}{\ensuremath{\mathbb{A}}}

\newcommand{\X}{\ensuremath{\mathscr{X}}}

\renewcommand{\A}{\ensuremath{\mathbb{A}}}
\renewcommand{\X}{\ensuremath{\mathfrak{X}}}

\newcommand{\Spec}{\ensuremath{\mathrm{Spec}\,}}

\newcommand{\Aut}{\ensuremath{\mathrm{Aut}}}

\newcommand{\vart}{\ensuremath{t}}

\numberwithin{equation}{section} \hyphenpenalty=6000
\begin{document}
\title{Geometric criteria for tame ramification}
\author[Johannes Nicaise]{Johannes Nicaise}
\address{KULeuven\\
Department of Mathematics\\
Celestijnenlaan 200B\\3001 Heverlee \\
Belgium} \email{johannes.nicaise@wis.kuleuven.be}

\thanks{The
research for this paper was partially supported by
ANR-06-BLAN-0183 and ANR-07-JCJC-0004. }
\begin{abstract}
We prove an A'Campo type formula for the tame monodromy zeta
function of a smooth and proper variety over a discretely valued
field $K$. As a first application, we relate the orders of the
tame monodromy eigenvalues on the $\ell$-adic cohomology of a
$K$-curve to the geometry of a relatively minimal $sncd$-model,
and we show that the semi-stable reduction theorem and Saito's
criterion for cohomological tameness are immediate consequences of
this result. As a second application, we compute the error term in
the trace formula for smooth and proper $K$-varieties. We see that
the validity of the trace formula would imply a partial
generalization of Saito's criterion to arbitrary dimension.
\end{abstract}

 \maketitle
\section{Introduction}
Let $R$ be a henselian discrete valuation ring with quotient field
$K$ and algebraically closed residue field $k$, and let $\ell$ be
a prime number different from the characteristic of $k$. We denote
by $K^t$ a tame closure of $K$. Let $X$ be a smooth and proper
$K$-variety. In Section \ref{sec-ac}, we compute the zeta function
$\zeta_X(\vart)$ of the tame monodromy action on the graded tame
$\ell$-adic cohomology
$$H(X\times_K K^t,\Q_\ell)=\bigoplus_{m\geq 0} H^m(X\times_K K^t,\Q_\ell)$$
 in terms of an $sncd$-model $\mathscr{X}$ of $X$ over $R$
(Theorem \ref{thm-ac}). This zeta function completely determines
the class of $H(X\times_K K^t,\Q_\ell)$ in the Grothendieck ring
 of $\ell$-adic representations of the tame
inertia group $G(K^t/K)$. Our formula for the monodromy zeta
function is an arithmetic analog of a formula obtained by A'Campo
 \cite{A'C} for
the zeta function of the monodromy action on the cohomology of the
Milnor fiber at a complex hypersurface singularity. The main
additional complication in the arithmetic setting is that we need
to prove a tameness property of the complex of $\ell$-adic tame
nearby cycles associated to $\mathscr{X}$ (Proposition
\ref{prop-tamenearby}). This property allows us to compute
$\zeta_X(\vart)$ pointwise on the special fiber $\mathscr{X}_s$ of
$\mathscr{X}$, using Grothendieck's description of the stalks of
the tame nearby cycles on a divisor with normal crossings
\cite[I.3.3]{sga7a}.

We present two applications of our arithmetic A'Campo formula. In
Section \ref{sec-saito}, we consider the case where $X$ is a
$K$-curve $C$. In Theorem \ref{thm-primed} and Corollary
\ref{cor-primed}, we relate the orders of the tame monodromy
eigenvalues on
$$H^1(C\times_K K^t,\Q_\ell)$$ to the geometry of the special
fiber $\mathscr{C}_s$ of a relatively minimal $sncd$-model
$\mathscr{C}$ of $C$. We show how Saito's criterion for
cohomological tameness (Theorem \ref{thm-saito}) and the
semi-stable reduction theorem (Theorem \ref{thm-sstable} and
Corollary \ref{cor-sstable}) are immediate consequences of this
result. Our methods also allow to determine the degree of the
minimal extension of $K$ where $C$ acquires semi-stable reduction,
if $C$ is cohomologically tame (Corollary \ref{cor-mindegree}).
Our approach is similar in spirit to the one in Saito's paper
\cite{saito}, but our proof substantially simplifies the
combinatorial analysis of $\mathscr{C}_s$. For different proofs of
Saito's criterion, see \cite{saito-log,stix} (using logarithmic
geometry) and \cite{halle} (using a geometric analysis of the
behaviour of $sncd$-models under base change). For a survey on the
semi-stable reduction theorem for curves, see \cite{abbes}.


As a second application, in Section \ref{sec-trace}, we compute
the error term in the trace formula for $X$ on an $sncd$-model
$\mathscr{X}$ of $X$ over $R$. The trace formula was introduced in
\cite{NiSe} and further studied in
\cite{Ni-trace,Ni-abelian,Ni-tracevar}. It expresses a certain
measure for the set of rational points on $X$ in terms of the
Galois action on the tame $\ell$-adic cohomology of $X$. We expect
that the trace formula is valid if $X$ is geometrically connected
and cohomologically tame, and $X(K^t)$ non-empty. We've proven
this if $k$ has characteristic zero \cite[6.5]{Ni-tracevar}, if
$X$ is a curve \cite[\S7]{Ni-tracevar}, and if $X$ is an abelian
variety \cite[2.9]{Ni-abelian}. Our formula for the error term
shows that (assuming the existence of an $sncd$-model), our
conjecture is equivalent to a partial generalization of Saito's
criterion to arbitrary dimension (Question \ref{ques-saito}).
\subsection*{Acknowledgements} I am grateful to L. Illusie and T.
Saito for suggesting a proof of Lemma \ref{lemm-tamenearby} and
Proposition \ref{prop-tamenearby}, and to L. Halvard Halle for
pointing out an error in a preliminary version of the paper.

\subsection*{Notations} Let $R$ be a henselian discrete valuation
ring, with quotient field $K$ and algebraically closed residue
field $k$. We fix a uniformizer $\pi$ in $R$. We denote by $p\geq
0$ the characteristic of $k$, and we fix a prime $\ell$ different
from $p$. We fix a separable closure $K^s$ of $K$, and we denote
by $K^t$ the tame closure of $K$ in $K^s$. We denote by $P$ the
wild inertia subgroup of $G(K^s/K)$, and we choose a topological
generator $\varphi$ of the tame inertia group $G(K^t/K)$. We
denote by $\N'$ the set of strictly positive integers that are not
divisible by $p$. We fix an algebraic closure $\Q_\ell^a$ of
$\Q_\ell$, and we denote by $\Q^a$ the algebraic closure of $\Q$
in $\Q_\ell^a$.

 If $X$ is a separated scheme of finite type over $K$, then we have
a canonical $G(K^t/K)$-equivariant isomorphism
$$H^m(X\times_K K^t,\Q_\ell)\cong H^m(X\times_K K^s,\Q_\ell)^P$$ for
every integer $m\geq 0$. We say that $X$ is cohomologically tame
if $P$ acts trivially on $H^m(X\times_K K^s,\Q_\ell)$ for all
$m\geq 0$.

If $Y$ is a separated scheme of finite type over a field $F$ and
$p'$ is a prime different from the characteristic of $F$, then we
denote by $\chi(Y)$ the $p'$-adic Euler characteristic (with
proper supports) of $Y$:
$$\chi(Y)=\sum_{m\geq 0}(-1)^m \mathrm{dim}\,H^m_c(Y\times_F F^s,\Q_{p'})$$
with $F^s$ a separable closure of $F$. It is well-known that the
Euler characteristic $\chi(Y)$ is independent of $p'$: if $F$ is a
finite field, then by the Grothendieck-Lefschetz trace formula,
$\chi(Y)$ equals minus the degree of the Hasse-Weil zeta function
of $Y$ \cite[1.5.4]{weil1}. The general case follows by a
spreading out argument and proper base change (if $F$ has
characteristic zero, it can also be deduced from the comparison
with singular cohomology). The Euler characteristic $\chi(Y)$ is
equal to the Euler characteristic without supports, i.e.,
$$\chi(Y)=\sum_{m\geq 0}(-1)^m \mathrm{dim}\,H^m(Y\times_F F^s,\Q_{p'}).$$
If $F$ has characteristic zero this result is due to Grothendieck;
the general case was proven by Laumon \cite{laumon}.

We denote by
$$(\cdot)_s:(\mathrm{Sch}/R)\rightarrow (\mathrm{Sch}/k):\mathscr{X}\mapsto \mathscr{X}_s=\mathscr{X}\times_R k$$
the special fiber functor from the category of $R$-schemes to the
category of $k$-schemes. We denote by $(\cdot)_{\mathrm{red}}$ the
endofunctor on the category of schemes that maps a scheme $S$ to
its maximal reduced closed subscheme $S_{\mathrm{red}}$. All
regular schemes are assumed to be locally Noetherian. When we
speak of a local ring $(A,\mathfrak{m}_A,k_A)$, we mean that $A$
is a local ring with maximal ideal $\mathfrak{m}_A$ and residue
field $k_A$.

For every integer $d>0$, we denote by $\Phi_d(\vart)\in \Z[\vart]$
the cyclotomic polynomial whose roots are the primitive $d$-th
roots of unity.
\section{The tame monodromy zeta function}\label{sec-ac}
\subsection{Models}\label{subsec-models}
We recall some standard definitions and fix some terminology. All
of the results in this section are well-known, but we include them
here for lack of suitable reference. All definitions, results and
proofs in Section \ref{subsec-models} are formulated in such a way
that they are valid over an arbitrary discrete valuation ring
$R$.

 Let $\mathscr{X}$ be a regular flat $R$-scheme, and let $x$ be a point of the special fiber
 $\mathscr{X}_s$. We say that $\mathscr{X}_s$ has strict normal
 crossings at $x$ if there exist a
regular system of parameters $(x_1,\ldots,x_m)$ and a unit $u$ in
$\mathcal{O}_{\mathscr{X},x}$ and elements $N_1,\ldots,N_m$ in
$\N$ such that \begin{equation}\label{eq-sncd}\pi=u \prod_{i=1}^m
(x_i)^{N_i}.\end{equation} Since every regular local ring is a
UFD, this is equivalent to saying that every tuple of
non-associated prime factors of $\pi$ in
$\mathcal{O}_{\mathscr{X},x}$ is part of a regular system of
parameters. We say that $\mathscr{X}$ is strictly semi-stable at
$x$ if $\mathscr{X}_s$ has strict normal crossings at $x$ and the
local ring $\mathcal{O}_{\mathscr{X}_s,x}$ is reduced. This is
equivalent to the property that each exponent $N_i$ in
\eqref{eq-sncd} is either zero or one.

 An  $sncd$-model over $R$ is a regular flat
separated $R$-scheme of finite type $\mathscr{X}$ such that
$\mathscr{X}_s$ has strict normal crossings at every point of
$\mathscr{X}_s$. An $sncd$-model $\mathscr{X}$ is called
semi-stable if $\mathscr{X}_s$ is reduced. This is equivalent to
the property that $\mathscr{X}$ is strictly semi-stable at every
point of $\mathscr{X}_s$.

\begin{lemma}\label{lemm-sncdequiv}
Let $\mathscr{X}$ be a regular flat $R$-scheme, and let $x$ be a
point of $\mathscr{X}_s$. We denote by $d$ the dimension of
$\mathscr{X}$ at $x$. Let $E_1,\ldots,E_n$ be the irreducible
components of $\mathscr{X}_s$ that pass through $x$, endowed with
their induced reduced structure. For every non-empty subset $J$ of
$\{1,\ldots,n\}$, we denote by $E_J$ the schematic intersection of
the closed subschemes $E_j$ of $\mathscr{X}$ with $j$ in $J$.
 Then the following
properties are equivalent.
\begin{enumerate}
\item The special fiber $\mathscr{X}_s$ has strict normal
crossings at $x$. \item For every non-empty subset $J$ of
$\{1,\ldots,n\}$, the scheme $E_J$ is regular
 and of dimension $d-|J|$ at $x$.
\end{enumerate}
\end{lemma}
\begin{proof}
%
Let $(x_1,\ldots,x_n)$ be a maximal tuple of non-associated prime
factors of $\pi$ in $\mathcal{O}_{\mathscr{X},x}$. Then the
irreducible components of $\Spec \mathcal{O}_{\mathscr{X}_s,x}$
are precisely the integral closed subschemes $\Spec
(\mathcal{O}_{\mathscr{X},x}/(x_i))$ of $\Spec
\mathcal{O}_{\mathscr{X},x}$, with $i\in \{1,\ldots,n\}$.
 It
follows from \cite[7.2.8.1]{ega1}, \cite[16.3.7 and
17.1.7]{ega4.1} and \cite[5.1.9]{ega4.2} that $(x_1,\ldots,x_n)$
is part of a regular system of parameters in
$\mathcal{O}_{\mathscr{X},x}$ if and only if condition (2) is
satisfied.
\end{proof}
\begin{corollary}\label{cor-localprop}
Let $\mathscr{X}$ be a regular flat $R$-scheme locally of finite
type. The set of points of $\mathscr{X}_s$ where $\mathscr{X}_s$
has strict normal crossings is open in $\mathscr{X}_s$.
\end{corollary}
\begin{proof}
 Let $x$ be a point of $\mathscr{X}_s$ such that
$\mathscr{X}_s$ has strict normal crossings at $x$, and denote by
$d$ the dimension of $\mathscr{X}$ at $x$. Let
 $E_1,\ldots,E_n$ be as in Lemma \ref{lemm-sncdequiv}. Replacing $\mathscr{X}$ by a suitable open
neighbourhood of $x$, we may assume that $E_1,\ldots,E_n$ are the
only irreducible components of $\mathscr{X}_s$.

Choose a non-empty subset $J$ of $\{1,\ldots,n\}$. Then $E_J$ is
regular and of dimension $d-|J|$ at $x$, by Lemma
\ref{lemm-sncdequiv}. The regular locus of $E_J$ is open in $E_J$
because $E_J$ is locally of finite type over the field $k$
\cite[6.12.5]{ega4.2}. Thus, shrinking $\mathscr{X}$, we may
assume that $E_J$ is regular for every non-empty subset $J$ of
$\{1,\ldots,n\}$. We may also assume that $E_J$ has pure dimension
$d-|J|$. Then $\mathscr{X}_s$ has strict normal crossings at every
point of $\mathscr{X}_s$, by Lemma \ref{lemm-sncdequiv}.
\end{proof}

\begin{lemma}\label{lemma-para0}
Let $\varphi:(A,\mathfrak{m}_A,k_A)\to (B,\mathfrak{m}_B,k_B)$ be
a local homomorphism of regular local rings.
 Then the following are equivalent:
\begin{enumerate}
\item the morphism  $\varphi$ is flat, and
$\mathfrak{m}_AB=\mathfrak{m}_B$,  \item there exists a regular
system of parameters $(a_1,\ldots,a_m)$ in $A$ such that
$(\varphi(a_1),\ldots,\varphi(a_m))$ is a regular system of
parameters in $B$, \item a tuple $(a_1,\ldots,a_m)$ of elements in
$A$ is a regular system of parameters if and only if
$(\varphi(a_1),\ldots,\varphi(a_m))$ is a regular system of
parameters in $B$, \item the morphism of $k_B$-vector spaces
$$\psi:(\mathfrak{m}_A/\mathfrak{m}^2_A)\otimes_{k_A}k_B\to
\mathfrak{m}_B/\mathfrak{m}^2_B$$ induced by $\varphi$ is an
isomorphism.
\end{enumerate}
\end{lemma}
\begin{proof}
%
If (1) holds, then $A$ and $B$ have the same dimension by
\cite[13.B]{matsumura},
 so that (2) follows from (1). Conversely,
(2) implies immediately that $\mathfrak{m}_A B=\mathfrak{m}_B$,
and flatness of $\varphi$ follows from the local criterion for
flatness (in the formulation of \cite[6.9]{eisenbud}) by induction
on the dimension of $A$. Thus (1) and (2) are equivalent.  The
implication (3)$\Rightarrow$(2) is trivial. By
\cite[17.1.7]{ega4.1}, a tuple $(c_1,\ldots,c_m)$ of elements in
the maximal ideal of a regular local ring $(C,\mathfrak{m}_C,k_C)$
is a regular system of parameters in $C$
 if and only if
the residue classes of the elements $c_i$ in the $k_C$-vector
space $\mathfrak{m}_C/\mathfrak{m}^2_C$ form a basis.
This shows that (2)$\Rightarrow$ (4)$\Rightarrow$(3).
%
\if false To simplify notation, we will denote $\varphi(a_j)$ by
$b_j$, for every index $j$. We start with the proof of (1). Let
$\varphi$ be flat and unramified, and let $(a_1,\ldots,a_m)$ be a
regular system of parameters in $A$. Then $A$ and $B$ have the
same dimension \cite[I.3.17]{milne} and $(b_1,\ldots,b_m)$
generate the maximal ideal in $B$, so that $(b_1,\ldots,b_m)$ is a
regular system of parameters in $B$. Assume, conversely, that
$(a_1,\ldots,a_m)$ is a regular system of parameters in $A$ such
that $(b_1,\ldots,b_m)$ is a regular system of parameters in $B$.
This immediately implies that $\varphi$ is unramified, and it
remains to show that $\varphi$ is flat.

Now we prove (2). Let $a_1,\ldots,a_n$ be elements in $A$. If
$a_1,\ldots,a_n$ are part of a regular system of parameters in
$A$, then it follows immediately from (1) that $b_1,\ldots,b_n$
are part of a regular system of local parameters in $B$. Assume,
conversely, that $b_1,\ldots,b_n$ are part of a regular system of
 parameters in $B$. \fi
\end{proof}
\begin{lemma}\label{lemm-para}
Let $f:Y\to X$ be a morphism of schemes that is locally of finite
presentation. Let $y$ be a point of $Y$ and set $x=f(y)$. Assume
that $X$ is regular at $x$ and that $Y$ is regular at $y$. Then
$f$ is \'etale at $y$ if and only if the residue field at $y$ is a
finite separable extension of the residue field at $x$ and the
morphism $\mathcal{O}_{X,x}\to \mathcal{O}_{Y,y}$ satisfies the
equivalent properties of Lemma \ref{lemma-para0}.
\end{lemma}
\begin{proof}
This follows from the characterization of \'etale morphisms in
\cite[17.6.1(c')]{ega4.4}.
\end{proof}

The following proposition describes the local structure of
semi-stable $sncd$-models.
\begin{prop}\label{prop-sncd}
Assume that $k$ is perfect. Let $\mathscr{X}$ be a regular flat
$R$-scheme locally of finite type, and let $x$ be a closed point
of $\mathscr{X}_s$. Then the following are equivalent:
\begin{enumerate} \item the $R$-scheme $\mathscr{X}$ is
strictly semi-stable at $x$, \item the point $x$  admits an  open
neighbourhood $\mathscr{U}$ in $\mathscr{X}$ such that there exist
integers $m\geq n>0$ and an \'etale $R$-morphism
$$g:\mathscr{U}\to \mathscr{Y}=\Spec R[y_1,\ldots,y_m]/(\pi-\prod_{j=1}^n
y_j)$$ such that $g(x)=O$. Here $O$ denotes the origin in
$\mathscr{Y}_s\subset \A^m_k$.
\end{enumerate}
\end{prop}
\begin{proof}
Assume that $\mathscr{X}$ is strictly semi-stable at $x$. Then we
have an expression of the form \eqref{eq-sncd} in
$\mathcal{O}_{\mathscr{X},x}$, with each $N_i$ either zero or one.
Permuting the local parameters $x_i$ if necessary, we may assume
that there exists an element $n$ of $\{1,\ldots,m\}$ such that
$N_i=1$ for $i=1,\ldots,n$ and $N_i=0$ for $i>n$. We can also
arrange that $u=1$ by replacing $x_1$ by $ux_1$.

The local parameters $x_i$ are germs of regular functions on
$\mathscr{X}$, and we choose a connected open neighbourhood
$\mathscr{U}$ of $x$ in $\mathscr{X}$ such that $x_i$ is defined
on $\mathscr{U}$ for every $i$. Then, by equation \eqref{eq-sncd}
and our assumptions, there exists a unique morphism of $R$-schemes
$$g:\mathscr{U}\to \mathscr{Y}=\Spec
R[y_1,\ldots,y_m]/(\pi-\prod_{j=1}^n y_j)$$ such that
$x_i=y_i\circ g$ for every $i$.

Note that $g(x)$ is the origin $O$ in $\mathscr{Y}_s$, and that
the pullback by $g$ of the regular system of parameters
$(y_1,\ldots,y_n)$ in $\mathcal{O}_{\mathscr{Y},O}$ is the regular
system of parameters $(x_1,\ldots,x_n)$ in
$\mathcal{O}_{\mathscr{X},x}$. Since $k$ is perfect, we also know
that the residue field at $x$ is separable over the residue field
$k$ at $O$.
 Thus, by Lemma \ref{lemm-para}, the morphism $g$ is \'etale at $x$. Shrinking $\mathscr{U}$, we
may assume that $g$ is \'etale everywhere. This shows that (1)
implies (2).

Conversely, assume that $\mathscr{X}$ and $x$ satisfy (2).
 We put $x_i=y_i\circ g$ for $i=1,\ldots,m$.
Since $h$ is \'etale, we know by Lemma \ref{lemm-para}  that
$(x_1,\ldots,x_m)$ is a regular system of parameters in
$\mathcal{O}_{\mathscr{X},x}$, and this system satisfies the
equation $$\pi=\prod_{j=1}^n x_j.$$  Thus $\mathscr{X}$ is
strictly semi-stable at $x$.
\end{proof}

Let $\mathscr{X}$ be a regular flat $R$-scheme, and let $x$ be a
point of $\mathscr{X}_s$. We say that $\mathscr{X}_s$ has normal
crossings at $x$ if there exists an \'etale morphism of
$R$-schemes $h:\mathscr{Z} \to \mathscr{X}$  such that
$\mathscr{Z}_s$ has strict normal crossings at some point $z$ of
$h^{-1}(x)$. Note that $\mathscr{Z}$ is regular and $R$-flat since
$h$ is \'etale \cite[17.5.8 and 17.6.1]{ega4.4}. If $\mathscr{X}$
is locally of finite type over $R$, then it follows from Corollary
\ref{cor-localprop} that the locus of points of $\mathscr{X}_s$
where $\mathscr{X}_s$ has normal crossings is open in
$\mathscr{X}_s$, since the image of an \'etale morphism of
$R$-schemes $\mathscr{Z}\to \mathscr{X}$ is open in $\mathscr{X}$
\cite[2.4.6]{ega4.2}.

We say that $\mathscr{X}$ is semi-stable at $x$ if $\mathscr{X}_s$
has normal crossings at $x$ and, moreover,
$\mathcal{O}_{\mathscr{X}_s,x}$ is reduced. We call $\mathscr{X}$
an $ncd$-model if $\mathscr{X}$ is separated and of finite type
over $R$
 and $\mathscr{X}_s$ has normal crossings at every point of
$\mathscr{X}_s$.
 An $ncd$-model $\mathscr{X}$ is called semi-stable
if $\mathscr{X}_s$ is reduced. This is equivalent to the property
that $\mathscr{X}$ is semi-stable at every point of
$\mathscr{X}_s$.

 It follows from Proposition \ref{prop-sncd} that, if $k$ is perfect, the generic fiber
$\mathscr{X}\times_R K$ of a proper semi-stable $sncd$-model
$\mathscr{X}$ is smooth over $K$. This implies that, if $k$ is
perfect, the generic fiber of a proper semi-stable $ncd$-model is
also smooth over $K$, since it has a $K$-smooth \'etale cover. The
properness assumption is needed to ensure that every collection of
open subsets of $\mathscr{X}$ that covers $\mathscr{X}_s$ also
covers $\mathscr{X}$.

\begin{prop}\label{prop-hensel}
Let $\mathscr{X}$ be a regular flat $R$-scheme, and let $x$ be a
point of $\mathscr{X}_s$. Let $y$ be a geometric point centered at
$x$, and denote by $\mathscr{Y}$ the strict henselization of
$\mathscr{X}$ at $y$. Then $\mathscr{Y}$ is a regular flat
$R$-scheme, and $\mathscr{X}_s$ has normal crossings at $x$ if and
only if $\mathscr{Y}_s$ has strict normal crossings at $y$.
\end{prop}
\begin{proof}
The scheme $\mathscr{Y}$ is regular \cite[18.8.13]{ega4.4} and the
local homomorphism $\mathcal{O}_{\mathscr{X},x}\to
\mathcal{O}(\mathscr{Y})$ satisfies the equivalent properties of
Lemma \ref{lemma-para0}, by \cite[18.8.8(iii)]{ega4.4}. Assume
that $\mathscr{X}_s$ has normal crossings at $x$, and choose an
\'etale morphism of $R$-schemes $h:\mathscr{Z}\to \mathscr{X}$
such that $\mathscr{Z}_s$ has strict normal crossings at some
point $z$ of $h^{-1}(x)$. Then by \cite[18.8.4]{ega4.4}, there
exists a morphism of $\mathscr{X}$-schemes $\mathscr{Y}\to
\mathscr{Z}$ that maps $y$ to $z$. Applying Lemma \ref{lemm-para}
to $h$, we see that the local homomorphism
$\mathcal{O}_{\mathscr{Z},z}\to \mathcal{O}(\mathscr{Y})$ also
satisfies the equivalent properties of Lemma \ref{lemma-para0},
which immediately implies that $\mathscr{Y}_s$ has strict normal
crossings at $y$.

Suppose, conversely, that $\mathscr{Y}_s$ has strict normal
crossings at $y$, and choose an equation of type \eqref{eq-sncd}
in $\mathcal{O}(\mathscr{Y})$. By construction
\cite[18.8.7]{ega4.4}, the ring $\mathcal{O}(\mathscr{Y})$ is a
direct limit of local rings that are essentially \'etale over
$\mathcal{O}_{\mathscr{X},x}$, thus we can find such a local ring
$A$ such that $u,x_1,\ldots,x_m$ lift to $A$, $u$ is a unit in $A$
and the equality \eqref{eq-sncd} holds in $A$. The ring
$\mathcal{O}(\mathscr{Y})$ is also the strict henselization of $A$
at $y$, so that the tuple $(x_1,\ldots,x_m)$ is a regular system
of parameters in $A$ by Lemma \ref{lemma-para0}. Since $A$ is
essentially \'etale over $\mathcal{O}_{\mathscr{X},x}$, it follows
from \cite[8.8.2]{ega4.3} that we can find an \'etale
$\mathscr{X}$-scheme $\mathscr{Z}$ and a point $z$ of
$\mathscr{Z}$ lying over $x$ such that $A$ and
$\mathcal{O}_{\mathscr{Z},z}$ are isomorphic as
$\mathcal{O}_{\mathscr{X},x}$-algebras. Then $\mathscr{Z}_s$ has
strict normal crossings at $z$. It follows that $\mathscr{X}_s$
has normal crossings at $x$.
\end{proof}

 Let $\mathscr{X}$ be a regular flat $R$-scheme, and let $\mathscr{Y}\to \mathscr{X}$ be an \'etale morphism.
 Then, as was already mentioned above, $\mathscr{Y}$
 is regular  and $R$-flat. Let $y$ be a
 point of $\mathscr{Y}_s$, and denote by $x$ its image in
 $\mathscr{X}_s$. The following properties follow easily from
 Lemma \ref{lemm-para} and Proposition \ref{prop-hensel}:
\begin{itemize}
\item if $\mathscr{X}_s$ has strict normal crossings at $x$ then
$\mathscr{Y}_s$ has strict normal crossings at $y$; \item if
$\mathscr{X}_s$ is strictly semi-stable at $x$ then
$\mathscr{Y}_s$ is strictly semi-stable at $y$; \item
$\mathscr{Y}_s$ has normal crossings at
 $y$ if and only if $\mathscr{X}_s$ has normal crossings at $x$;
\item $\mathscr{Y}$ is semi-stable at $y$ if and only if
$\mathscr{X}$ is semi-stable at $x$.
\end{itemize}
(Note that $\mathscr{Y}_s$ is reduced at $y$ if and only if
$\mathscr{X}_s$ is reduced at $x$, by \cite[17.5.7]{ega4.4}.)

%

\begin{prop}\label{prop-ncd}
Let $\mathscr{X}$ be a regular flat $R$-scheme, and let $x$ be a
point of $\mathscr{X}_s$ such that $\mathscr{X}_s$ has normal
crossings at $x$. Then $\mathscr{X}_s$ has strict normal crossings
at $x$ if and only if every irreducible component of
$\mathscr{X}_s$ that passes through $x$ (endowed with its reduced
induced structure) is regular at $x$.
\end{prop}
\begin{proof}
If $\mathscr{X}_s$ has strict normal crossings at $x$, then every
irreducible component of $\mathscr{X}_s$ that passes through $x$
is regular at $x$, by Lemma \ref{lemm-sncdequiv}. Conversely,
assume that every irreducible component of $\mathscr{X}_s$ that
passes through $x$ is regular at $x$. 
Let $(x_1,\ldots,x_n)$ be a tuple of non-associated prime factors
of $\pi$ in $\mathcal{O}_{\mathscr{X},x}$. It's enough to show
that this tuple is part of a regular system of parameters in
$\mathcal{O}_{\mathscr{X},x}$. By \cite[18.8.8(iii)]{ega4.4} and
Lemma \ref{lemma-para0}, we can verify this in a strict
henselization $A$ of $\mathcal{O}_{\mathscr{X},x}$.

Locally at $x$, each of the equations $x_i=0$ defines an
irreducible component of $\mathscr{X}_s$, so that  $x_i$ is part
of a regular system of parameters by \cite[17.1.7]{ega4.1}.
  It follows from \cite[18.8.8(iii)]{ega4.4} and Lemma
\ref{lemma-para0} that the image of $x_i$ in $A$ is still part of
a regular system of parameters. In particular, this element is
prime. Moreover, \cite[4.C(ii)]{matsumura} implies that $x_i$ is
not associated to $x_j$ in $A$ if $i$ and $j$ are distinct
elements of $\{1,\ldots,n\}$, because $x_i$ and $x_j$ are not
associated in $\mathcal{O}_{\mathscr{X},x}$ and $A$ is faithfully
flat over $\mathcal{O}_{\mathscr{X},x}$. Thus $x_1,\ldots,x_n$ are
non-associated prime factors of $\pi$ in $A$.
 It follows that $(x_1,\ldots,x_n)$ is part of a regular system of parameters
in $A$, because $(\Spec A)_s$ has strict normal crossings at its
unique closed point, by Proposition \ref{prop-hensel}.
\end{proof}

\begin{definition}
If $X$ is a proper $K$-scheme, then a model of $X$ is a flat
proper $R$-scheme $\mathscr{X}$ endowed with an isomorphism of
$K$-schemes
$$\mathscr{X}\times_R K\rightarrow X.$$
We say that $\mathscr{X}$ is an $ncd$-model (resp. $sncd$-model)
of $X$ if, moreover, $\mathscr{X}$ is regular and $\mathscr{X}_s$
has normal crossings (resp. strict normal crossings) at every
point of $\mathscr{X}_s$. This implies that $X$ is regular.
\end{definition}
A morphism of models of $X$ is an $R$-morphism $h$ such that the
induced morphism $h_K$ between the generic fibers commutes with
the respective isomorphisms to $X$. In particular, $h_K$ is an
isomorphism.
 We say that a regular model
$\mathscr{X}$ of $X$ is {\em relatively minimal} if every morphism
to another regular model is an isomorphism. We say that
$\mathscr{X}$ is {\em minimal} if, up to isomorphism, it is the
unique relatively minimal  regular model. The analogous
terminology applies to $ncd$-models and $sncd$-models.
\subsection{The case of curves}
From now on, we'll assume that $R$ is henselian and that $k$ is
algebraically closed. To describe the geometry of models of
curves, we gather some results from \cite{liu}. Beware that
 the author of \cite{liu} uses the term ``normal crossings'' where
we use ``strict normal crossings'', see  \cite[9.1.6 and
9.1.7]{liu}. The key lemma for minimality issues of $sncd$-models
is \cite[9.3.35]{liu}. Unfortunately, this statement is not
entirely correct. In our situation, it can be corrected and
generalized as follows.

\begin{lemma}\label{lemm-nosncd}
Let $\mathscr{X}$ be a regular flat $R$-scheme of pure dimension
two. Let $x$ be a point of $\mathscr{X}_s$ where $\mathscr{X}_s$
has normal crossings. Then $\mathscr{X}_s$ has strict normal
crossings at $x$ unless $x$ lies on precisely one irreducible
component $\Gamma$ of $\mathscr{X}_s$ and $x$ is a singular point
of $\Gamma$ (with its reduced induced structure).
\end{lemma}
\begin{proof}
Since $\mathcal{O}_{\mathscr{X},x}$ has dimension two and
$\mathscr{X}_s$ has normal crossings at $x$, the point $x$ can lie
on at most two irreducible components of $\mathscr{X}_s$.  If $x$
lies on only one irreducible component $\Gamma$ of
$\mathscr{X}_s$, then $\mathscr{X}_s$ has strict normal crossings
at $x$ if and only if $\Gamma$ is regular at $x$, by Proposition
\ref{prop-ncd}. So we may assume that $x$ lies on two distinct
irreducible components of $\mathscr{X}_s$. Then $\pi$ has
precisely two non-associated prime factors $x_1$ and $x_2$ in
$\mathcal{O}_{\mathscr{X},x}$, and it is enough to show that
$(x_1,x_2)$ is a regular system of parameters in
$\mathcal{O}_{\mathscr{X},x}$.

Let $A$ be the henselization of $\mathcal{O}_{\mathscr{X},x}$. It
is a regular \cite[18.6.10]{ega4.4} and the local homomorphism
$\mathcal{O}_{\mathscr{X},x}\to A$ satisfies the equivalent
properties in Lemma \ref{lemma-para0}, by
\cite[18.6.6(iii)]{ega4.4}. Thus it suffices to prove that
$(x_1,x_2)$ is a regular system of parameters in $A$. We know by
Proposition \ref{prop-hensel} that $(\Spec A)_s$ has strict normal
crossings at $x$, so that we only have to show that $x_1$ and
$x_2$ are non-associated prime factors of $\pi$ in $A$. For
$i=1,2$, the ring $A/(x_i)$ is reduced because it is the
henselization of the reduced local ring
$\mathcal{O}_{\mathscr{X},x}/(x_i)$ \cite[18.6.8 and
18.6.9]{ega4.4}. Moreover, $x_1$ and $x_2$ have no common prime
factor in $A$ because $x_1$ is not a zero-divisor in
$B=\mathcal{O}_{\mathscr{X},x}/(x_2)$ so that it cannot be a zero
divisor in the faithfully flat $B$-algebra $A/(x_2)$. It follows
that $x_1$ and $x_2$ are
 the two non-associated prime factors of $\pi$ in $A$.
\end{proof}

\begin{prop}\label{prop-liucor}
Let $\mathscr{X}$ be a regular flat $R$-scheme of pure dimension
two. Let $x$ be a closed point of $\mathscr{X}_s$. We denote by
$f:\mathscr{X}'\to \mathscr{X}$ the blow-up of $\mathscr{X}$ at
$x$, and by $E=f^{-1}(x)$ the exceptional divisor of $f$.
\begin{enumerate}
\item If $\mathscr{X}_s$ has normal crossings at $x$, then
$\mathscr{X}'_s$ has strict normal crossings at every point of
$E$, and $E$ meets the other irreducible components of
$\mathscr{X}'_s$ in at most two points.

 \item Assume that $\mathscr{X}'_s$ has normal crossings at every point of $E$. Then $\mathscr{X}_s$ has normal crossings at $x$ if and
only if $E$ meets the other irreducible components of
$\mathscr{X}'_s$ in at most two points.

 \item Assume that $\mathscr{X}'_s$
has normal crossings at every point of $E$. Then $\mathscr{X}_s$
has strict normal crossings at $x$ if and only if $E$ meets the
other components of $\mathscr{X}'_s$ in at most two points and $E$
does not intersect any other component twice.
\end{enumerate}
\end{prop}
\begin{proof}
Even though we are not dealing with proper $R$-schemes, we can
still borrow most of the arguments from \cite{liu}, using the
intersection theory on regular two-dimensional schemes developed
in \cite{lichtenbaum}.

First, we prove (1). We choose an \'etale morphism of $R$-schemes
$\mathscr{Y}\to \mathscr{X}$ such that $\mathscr{Y}_s$ has strict
normal crossings at some point $y$ lying over $x$. Then we can
apply \cite[9.2.31]{liu} to the blow-up of $\mathscr{Y}$ at $y$.
Since blowing up commutes with flat base change
\cite[8.1.12]{liu}, it follows that $\mathscr{X}'_s$ has normal
crossings at every point of $E$ and that $E$ intersects the other
components of $\mathscr{X}'_s$ in at most two points. The
exceptional curve $E$ is regular, so that Lemma \ref{lemm-nosncd}
implies that $\mathscr{X}'_s$ has strict normal crossings at every
point of $E$.

Now we prove (3). We only indicate where the proof of
\cite[9.3.35]{liu} must be modified.  On line 5 of the proof, it
is tacitly assumed that $\widetilde{\Gamma}$ intersects $E$ in at
most one point. This is not always the case under the hypotheses
of \cite[9.3.35]{liu}; we added it as an assumption in (3). The
argument in \cite[9.3.35]{liu} can be copied  verbatim to prove
point (3) of our proposition.

Finally, we prove (2). Assume that $E$ verifies the conditions in
the statement. We will prove that $\mathscr{X}_s$ has normal
crossings at $x$. The converse implication follows from (1). We
may assume that $\mathscr{X}$ is a strictly henselian local
scheme, by Proposition \ref{prop-hensel} and the fact that blowing
up commutes with flat base change \cite[8.1.12]{liu}. Then it
follows from \cite[18.6.8]{ega4.4} that $E$ cannot meet any other
irreducible component of $\mathscr{X}'_s$ twice, so that we can
deduce from (3) that $\mathscr{X}_s$ has strict normal crossings
at $x$.
\end{proof}

Let $C$ be a smooth, proper, geometrically connected $K$-curve of
genus $g$. The curve $C$ admits a relatively
 minimal regular model $\mathscr{C}$ \cite[10.1.8]{liu}, and every regular model of $C$
admits a morphism to some relatively minimal regular model
 \cite[9.3.19]{liu}. This morphism is a composition
of contractions of irreducible components in the special fiber.

Assume that $g\geq 1$. Then $\mathscr{C}$ is minimal
\cite[9.3.21]{liu}. Repeatedly blowing up $\mathscr{C}$ at points
where $\mathscr{C}_s$ does not have normal crossings, we obtain a
minimal $ncd$-model $\mathscr{C}'$ of $C$.
 Blowing up $\mathscr{C}'$ at the
self-intersection points of the irreducible components of its
special fiber, we obtain a minimal $sncd$-model. This can be
proved as in \cite[9.3.36]{liu}, invoking Proposition
\ref{prop-liucor} instead of \cite[9.3.35]{liu}.

Now suppose that $g=0$. This case is treated in
\cite[pp.\,155--157]{shafarevich} and \cite[Exercise 9.3.1]{liu}.
Under our assumptions ($R$ henselian and $k$ algebraically
closed), the Brauer group of $K$ is trivial \cite[6.2]{dixexp}, so
that the conic $C$ is isomorphic to $\mathbb{P}^1_K$. The
$R$-scheme $\mathbb{P}^1_R$ is a relatively minimal regular model
of $\mathbb{P}^1_K$ which is not minimal. It is also a relatively
minimal $ncd$-model and $sncd$-model of $\mathbb{P}^1_K$.
Moreover, every relatively minimal regular model of
$\mathbb{P}^1_K$ is smooth over $R$, and its special fiber is
isomorphic to $\mathbb{P}^1_k$.


\subsection{Constructions on $sncd$-models}\label{subsec-sncd}
Let $\mathscr{X}$ be an $sncd$-model over $R$. We put
$X=\mathscr{X}\times_R K$. We write
$$\mathscr{X}_s=\sum_{i\in I}N_i E_i$$
where $E_i,\,i\in I$ are the irreducible components of
$\mathscr{X}_s$, and $N_i$ is the multiplicity of $E_i$ in the
Cartier divisor $\mathscr{X}_s$ on $\mathscr{X}$. For each $i\in
I$, we denote by $N_i'$ the largest divisor of $N_i$ that is not
divisible by $p$. Note that $N'_i=N_i$ if $p=0$.

Let $J$ be a non-empty subset of $I$. We set
$$N'_J=\gcd\{N'_j\,|\,j\in J\}.$$
Moreover, we put \begin{eqnarray*} E_J&=&\bigcap_{j\in J}E_j
\\E_J^o&=&E_J\setminus (\bigcup_{i\in I\setminus J}E_i).
\end{eqnarray*}
The set
$$\{E_J^o\,|\,\emptyset\neq J\subset I\}$$
is a partition of $\mathscr{X}_s$ into locally closed subsets. We
endow all $E_J$ and $E_J^o$ with their reduced induced structures.
 The
schemes $E_J$ and $E_J^o$ are regular, by Lemma
\ref{lemm-sncdequiv}. For every $i\in I$, we write $E_i^o$ instead
of $E^o_{\{i\}}$.

\begin{lemma}\label{lemma-cov}
For every non-empty subset $J$ of $I$, there exist integral affine
open subschemes $\mathscr{U}_1,\ldots,\mathscr{U}_r$ of
$\mathscr{X}$ such that
\begin{itemize}
\item $E_J^o$ is contained in $\mathscr{U}=\cup_{i=1}^r
\mathscr{U}_i$, \item on each open subscheme $\mathscr{U}_i$, we
can write $\pi=u_i (v_i)^{N'_J}$ with $u_i,\,v_i$ regular
functions on $\mathscr{U}_i$ such that $u_i$ a unit.
\end{itemize}
\end{lemma}
\begin{proof}
Let $x$ be a closed point of $E_J^o$. Since $\mathscr{X}$ is an
$sncd$-model, we can find a regular system of parameters
$(x_1,\ldots,x_m)$ and a unit $u$ in $\mathcal{O}_{\mathscr{X},x}$
such that
$$\pi=u\prod_{j=1}^m(x_j)^{M_j}$$ for some $M_1,\ldots,M_m$ in
$\N$. Permuting the parameters $x_j$, we may assume that there
exists an $n\in \{1,\ldots,m\}$ such that $M_j>0$ for
$j=1,\ldots,n$ and $M_j=0$ for $j>n$. The irreducible components
of $\mathscr{X}_s$ that pass through $x$ are the components $E_i$
with $i\in J$, and they are locally defined by the equations
$x_j=0$, for $j=1,\ldots,n$. This correspondence yields a
bijection between the set $J$ and the set $\{1,\ldots,n\}$. Modulo
this identification, we have $N_j=M_j$ for every $j\in J$.

The elements $u$ and $x_1,\ldots,x_m$ are germs of regular
functions on $\mathscr{X}$, and we choose an affine integral open
neighbourhood $\mathscr{V}$ of $x$ in $\mathscr{X}$ such that
 $u$ and $x_1,\ldots,x_m$ are all defined on $\mathscr{V}$ and $u$ is a unit in $\mathcal{O}(\mathscr{V})$. Then we have the
 equation
 $$\pi=u\prod_{j\in J}(x_j)^{N_j}$$ in $\mathcal{O}(\mathscr{V})$. Writing
 $$v= \prod_{j\in J}(x_j)^{N_j/N'_J},$$ we obtain the equation
 $\pi=u v^{N'_J}$. Therefore, $E_J^o$ can be covered by finitely many open subschemes
 $\mathscr{U}_i$ of $\mathscr{X}$ as in the statement of the lemma.
\end{proof}

We  keep the notations of Lemma \ref{lemma-cov}. We write
$$\mathscr{U}_i=\Spec A_i$$ for $i=1,\ldots,r$, and we define a
finite \'etale covering of $\Spec A_i$ by
\begin{equation}\label{eq-localcover}\mathscr{V}_i=\Spec A_i[t_i]/((t_i)^{N'_J}-u_i)\rightarrow
\mathscr{U}_i.\end{equation} These coverings glue to a finite
\'etale covering
\begin{equation}\label{eq-covering}\widetilde{\mathscr{U}}\rightarrow \mathscr{U}\end{equation} of degree $N'_J$, the gluing
data being given by $t_i=v_j t_j/v_i$ over
$\mathscr{U}_{ij}:=\mathscr{U}_i\cap \mathscr{U}_j$ (note that
$v_j /v_i$ is regular
 on $\mathscr{V}_j\times_{\mathscr{U}_j} \mathscr{U}_{ij}$, because this scheme is
normal, and $(v_j /v_i)^{N'_J}=u_i/u_j\in
\mathcal{O}(\mathscr{U}_{ij})$).
 We put
$$\widetilde{E}^o_J=\widetilde{\mathscr{U}}\times_\mathscr{U} E_J^o.$$
This is a finite \'etale covering of $E_J^o$ of degree $N'_J$. Up
to $E_J^o$-isomorphism, it is independent of the choices of
$\mathscr{U}_i$, $u_i$ and $v_i$. In fact, we have the following
alternative construction.

\begin{prop}\label{prop-normal}
Let $J$ be a non-empty subset of $I$, and denote by $\mathscr{Y}$
the normalization of
$$\mathscr{X}\times_{R}(R[s]/(s^{N'_J}-\pi)).$$
 Then
 $\widetilde{E}_J^o$ and $\mathscr{Y}\times_{\mathscr{X}}E_J^o$ are isomorphic as
 $E_J^o$-schemes.
\end{prop}
\begin{proof}
We set $R'=R[s]/(s^{N'_J}-\pi)$ and
$\mathscr{X}'=\mathscr{X}\times_R R'$. We denote by $K'$ the
quotient field of $R'$. It is a finite separable extension of $K$.
The morphism $$\mathscr{Y}\times_{R'} K'\to
\mathscr{X}'\times_{R'} K'\cong \mathscr{X}\times_R K'$$ is an
isomorphism, because $\mathscr{X}\times_R K$ is regular so that
$\mathscr{X}\times_R K'$ is regular \cite[6.7.4]{ega4.2}, and thus
normal.

 It is enough to show that, in the notation
of \eqref{eq-covering}, $\widetilde{\mathscr{U}}$ is isomorphic to
$\mathscr{Y}\times_\mathscr{X} \mathscr{U}$ as a
$\mathscr{U}$-scheme. Since normalization commutes with open
immersions, we may assume that $\mathscr{U}=\mathscr{X}$.

The scheme $\widetilde{\mathscr{U}}$ is regular and $R$-flat,
because $\widetilde{\mathscr{U}}\to \mathscr{U}$ is \'etale and
$\mathscr{U}$ is regular and $R$-flat. In particular,
$\widetilde{\mathscr{U}}$ is normal. The elements $t_iv_i\in
\mathcal{O}(\mathscr{V}_i)$ glue to a regular function $w$ on
$\widetilde{\mathscr{U}}$. We have $w^{N'_J}=\pi$ on
$\widetilde{\mathscr{U}}$ because this holds on every open
$\mathscr{V}_i$. There is a unique morphism of
$\mathscr{X}$-schemes
$$g:\widetilde{\mathscr{U}}\to \mathscr{X}'$$
such that $s\circ g=w$, and it factors uniquely through a morphism
$$h:\widetilde{\mathscr{U}}\to \mathscr{Y}$$ because $\widetilde{\mathscr{U}}$ is
normal. One sees from the local description in
\eqref{eq-localcover} that the induced morphism
$$h_{K'}:\widetilde{\mathscr{U}}\times_{R'} K'\to \mathscr{Y}\times_{R'} K'\cong \mathscr{X}\times_R K'$$ is an
isomorphism, since $\pi=u_i v_i^{N'_J}$ and $v_i$ is a unit on
$\mathscr{V}_i\times_R K$ for every $i$ in $\{1,\ldots,r\}$. Thus
$h$ is birational, because $\widetilde{\mathscr{U}}$ and
$\mathscr{Y}$ are $R$-flat. Moreover, $\widetilde{\mathscr{U}}\to
\mathscr{U}$ is finite, so that $h$ is finite. Since $\mathscr{Y}$
is normal, we can conclude by \cite[4.4.9]{ega3.1} that $h$ is an
isomorphism.
\end{proof}
\subsection{Tame nearby cycles}\label{subsec-tamecons}
Let $\mathscr{Y}$ be a separated $R$-scheme of finite type. Let
$\Lambda$ be either $\Q_\ell$, or $\Z_\ell$, or a Noetherian
torsion ring that is killed by an element of $\N'$. We denote by
$D^b_c(\mathscr{Y}_s,\Lambda)$ the bounded derived category of
 constructible sheaves of $\Lambda$-modules on $\mathscr{Y}_s$. If $\Lambda$ is a
 torsion ring this is simply the full subcategory of the derived
 category of \'etale sheaves of $\Lambda$-modules on
 $\mathscr{Y}_s$ consisting of complexes with bounded and
 constructible cohomology. If $\Lambda$ is $\Q_\ell$ or $\Z_\ell$
 the definition is more delicate; see \cite[1.1.2]{deligne-weilII}
 (note that the finiteness conditions in c) en d) of \cite[1.1.2]{deligne-weilII} are fulfilled, since $\mathscr{Y}_s$ is of finite type
 over the algebraically closed field $k$).

We denote by
$$R\psi_{\mathscr{Y}}(\Lambda)\mbox{ and }R\psi^t_{\mathscr{Y}}(\Lambda) \in D^b_c(\mathscr{Y}_s,\Lambda)$$ the complex of nearby
cycles, resp. tame nearby cycles, with coefficients in $\Lambda$
associated to $\mathscr{Y}$. If $\Lambda$ is torsion, these
objects were defined in \cite[Exp.\,I]{sga7a} and
\cite[Exp.\,XIII]{sga7b}, and the fact that
$R\psi_{\mathscr{Y}}(\Lambda)$ is constructible  was proven in
\cite[Th.\,finitude(3.2)]{sga4.5}.  It follows that
$R\psi^t_{\mathscr{Y}}(\Lambda)$ is constructible, because
\begin{equation}\label{eq-invar}
R^i\psi^t_{\mathscr{Y}}(\Lambda)\cong
(R^i\psi_{\mathscr{Y}}(\Lambda))^P\end{equation} for every $i$ in
$\N$ \cite[I.2.7.2]{sga7a}.

For every integer $n>0$, the object
$R\psi_{\mathscr{Y}}(\Z/\ell^n)$ has finite Tor-dimension, and it
is compatible with reduction of the coefficients modulo powers of
$\ell$ \cite[D.8]{kiehl-weissauer}. Thus we can define the object
$R\psi_{\mathscr{Y}}(\Lambda)$ in $D^b_c(\mathscr{Y}_s,\Lambda)$
 when $\Lambda=\Z_\ell$ or $\Lambda=\Q_{\ell}$ by passing to the limit; see
 \cite[1.1.2(c)]{deligne-weilII} and
 \cite[p.\,354]{kiehl-weissauer}.

  Let $M$ be a $(\Z/\ell^n)$-module with
 continuous $P$-action.
 Since $P$ is a pro-$p$-group and $p$
 is different from $\ell$, the module $M^P$ is a direct summand of
 $M$.
 It is split off by the averaging map $$M\to M^P:m\mapsto \frac{1}{|P/P_m|}\sum_{g\in
 P/P_m}g\cdot m$$ where $P_m$ denotes the stabilizer of $m$, which is an open subgroup of $P$ and thus of finite index.
 It follows that
 $$(M\otimes_{\Z/\ell^n}\Z/\ell^m)^P\cong
 M^P\otimes_{\Z/\ell^n}\Z/\ell^m$$ for all integers $n\geq m>0$
 and that the functor $(\cdot)^P$ is exact on the category of
 $(\Z/\ell^n)$-modules with continuous $P$-action.

Using these properties, we deduce from \eqref{eq-invar}
 that $R\psi^t_{\mathscr{Y}}(\Z/\ell^n)$ has finite Tor-dimension
 for every integer $n>0$ and
 that $R\psi^t_{\mathscr{Y}}$ is compatible with reduction of the coefficients modulo powers of
$\ell$. Thus, we can define $R\psi^t_{\mathscr{Y}}(\Lambda)$ in
$D^b_c(\mathscr{Y}_s,\Lambda)$
 when $\Lambda=\Z_\ell$ or $\Lambda=\Q_{\ell}$ by passing to the limit, and we still have
an isomorphism \eqref{eq-invar} in those cases.

\subsection{Tame nearby cycles on divisors with strict normal
crossings} We keep the notations of Section \ref{subsec-tamecons}.
 The following lemma and proposition constitute the key
technical result of this section. The proofs were suggested to me
by L. Illusie and T. Saito.

\begin{lemma}\label{lemm-tamenearby}
Let $\mathscr{Y}$ be a regular flat separated $R$-scheme of finite
type, of pure dimension $n$.  Consider an integer $q$ in
$\{1,\ldots,n\}$ and a tuple $(M_1,\ldots,M_q)$ in $(\Z_{>0})^q$.
For each $i\in \{1,\ldots,q\}$, we denote by $M'_i$ the largest
divisor of $M_i$ that is not divisible by $p$. We put
$$\mu=\gcd\{M'_i\,|\,i\in \{1,\ldots,q\}\,\}.$$

Let $y$ be a closed point of $\mathscr{Y}_s$. Assume that there
exist a regular system of parameters $(y_1,\ldots,y_n)$ and a unit
$v$ in $\mathcal{O}_{\mathscr{Y},y}$ such that
\begin{equation}\label{eq-mu}
\pi=v^{\mu}\prod_{i=1}^q (y_i)^{M_i}.\end{equation} Then there
exists an integral affine open neighbourhood
$$\mathscr{U}=\Spec B$$ of $y$ in $\mathscr{Y}$ such that
$y_1,\ldots,y_q$ are regular functions on $\mathscr{U}$ and such
that for each $m\in \N$, the sheaf
$$R^m\psi^t_{\mathscr{Y}}(\Lambda)$$ is constant on the subscheme
$$U=\Spec (B/(y_1,\ldots,y_q))$$ of $\mathscr{Y}_s$.
\end{lemma}
\begin{proof}
It is enough to consider the case where $\Lambda$ is torsion. For
each $i\in \{1,\ldots,q\}$, we write
$$M_i=e_iM'_i$$ with $e_i\in \N$. If $p=0$ then all $e_i$ are equal to
one; if $p>1$ then all $e_i$ are powers of $p$.

Shrinking $\mathscr{Y}$, we may assume that
 $\mathscr{Y}$ is integral and affine, say,
$\mathscr{Y}=\Spec B$, and that $v$ and $y_1,\ldots,y_q$ are
regular functions on $\mathscr{Y}$, with $v$ a unit in $B$. Then
the equation \eqref{eq-mu} holds in $B$. We may also assume that
$\mathscr{Y}$ is an $sncd$-model, by Corollary
\ref{cor-localprop}, and that $y_i$ is a prime element in $B$ for
$i=1,\ldots,q$. We put
$$U=\Spec (B/(y_1,\ldots,y_q)).$$

 By B\'ezout's theorem, there
exist integers $\alpha_1,\ldots,\alpha_q$ such that
$$\mu=\sum_{i=1}^q \alpha_i M'_i.$$
 We put
$$\mathscr{Z}=\Spec R[z_1,\ldots,z_n]/(\pi-\prod_{i=1}^q
(z_i)^{M'_i})$$ and we consider the morphism
$f:\mathscr{Y}\rightarrow \mathscr{Z}$ defined by
\begin{eqnarray*}
 z_i&\mapsto & v^{\alpha_i}(y_i)^{e_i} \mbox{ for }i=1,\ldots,q,
\\ z_i&\mapsto & y_i \mbox{ for }i=q+1,\ldots,n.
\end{eqnarray*}
Then $f$ maps $y\in \mathscr{Y}_s$ to the origin in
$\mathscr{Z}_s$, and $f(U)$ is contained in the closed subscheme
$$V=\Spec k[z_{q+1},\ldots,z_n]$$
of $\mathscr{Z}_s$.

We denote by $\theta$ the base change morphism
\begin{equation}\label{eq-theta}
\theta: f_s^*R^m\psi^t_{\mathscr{Z}}(\Lambda)\rightarrow
R^m\psi^t_{\mathscr{Y}}(\Lambda)\end{equation} of
$\Lambda$-sheaves on $\mathscr{Y}_s$
\cite[XIII.2.1.7.2]{sga7b}. 
 We claim that  $\theta$ is an
isomorphism.
 Assuming this for now, it suffices to prove that
$R\psi^t_{\mathscr{Z}}(\Lambda)$ is constant on $V$. Consider the
morphism
$$g:\mathscr{Z}\rightarrow \mathscr{Z}'=\Spec
R[z'_1,\ldots,z'_q]/(\pi-\prod_{i=1}^q (z'_i)^{M'_i})$$ defined by
$$z'_i\mapsto z_i\mbox{ for }i=1,\ldots,q.$$
The morphism $g$ is smooth. By smooth base change, we have
$$R^m\psi^t_{\mathscr{Z}}(\Lambda)\cong
g_s^*R^m\psi^t_{\mathscr{Z}'}(\Lambda)$$ for each $m\in \N$
\cite[XIII.2.1.7.2]{sga7b}. Thus the restriction of
$R^m\psi^t_{\mathscr{Z}}(\Lambda)$ to $V=g_s^{-1}(0)$ is constant.

It remains to prove our claim. We'll use the local computations in
\cite[I.3.3]{sga7a} of the tame nearby cycles on a divisor with
strict normal crossings (these computations assume a purity
property that was later proven by Gabber \cite{gabber}).
 Let
$\overline{a}$ be a geometric point of $\mathscr{Y}_s$ and denote
by $\overline{b}$ its image $f\circ \overline{a}$ in
$\mathscr{Z}_s$. It is enough to show that, for all integers
$m\geq 0$, the morphism
\begin{equation}\label{eq-stalk}
R^m\psi^t_{\mathscr{Z}}(\Lambda)_{\overline{b}}\rightarrow
R^m\psi^t_{\mathscr{Y}}(\Lambda)_{\overline{a}}\end{equation}
obtained from \eqref{eq-theta} by passing to the stalks at
$\overline{a}$, is an isomorphism. We denote by
$\mathscr{Y}_{\overline{a}}$ and $\mathscr{Z}_{\overline{b}}$ the
strict localization of $\mathscr{Y}$ at $\overline{a}$, resp.
$\mathscr{Z}$ at $\overline{b}$, and we set
$Y=\mathscr{Y}_{\overline{a}}\times_R K$ and
$Z=\mathscr{Z}_{\overline{b}}\times_R K$. Then $f$ induces a
morphism of $K$-schemes $Y\to Z$ and we can identify
\eqref{eq-stalk} with the morphism
\begin{equation}\label{eq-strictloc}H^m(Z\times_K K^t,\Lambda)\to
H^m(Y\times_K K^t,\Lambda)\end{equation} (see
\cite[I.2.3]{sga7a}).

Denote by $I$ the subset of $\{1,\ldots,q\}$ consisting of indices
$i$ such that $y_i$ vanishes at $\overline{a}$, and set $\nu=\gcd
\{M'_i\,|\,i\in I\}$. We denote by $K'$ the unique degree $\nu$
extension of $K$ in $K^t$ (obtained by adding a $\nu$-th root of a
uniformizer) and by $R'$ the normalization of $R$ in $K'$.

Since every unit in $\mathcal{O}(\mathscr{Y}_{\overline{a}})$ is a
$\nu$-th power, the proof of Proposition \ref{prop-normal} shows
that the normalization of $\mathscr{Y}_{\overline{a}}\times_R R'$
is the disjoint union of $\nu$ copies of
$\mathscr{Y}_{\overline{a}}$, which are transitively permuted by
the Galois action of $G(K'/K)\cong \mu_{\nu}(k)$. The
$R'$-structure of such a copy $\mathscr{C}$ of
$\mathscr{Y}_{\overline{a}}$ is determined by the choice of a
$\nu$-th root of $\pi$ in
$\mathcal{O}(\mathscr{Y}_{\overline{a}})$. The special fiber of
the $R'$-scheme $\mathscr{C}$ is a divisor with strict normal
crossings with multiplicities $M_i/\nu$, $i\in I_{\overline{a}}$.
The generic fiber of the normalization of
$\mathscr{Y}_{\overline{a}}\times_R R'$ is canonically isomorphic
to $Y\times_K K'$.  A similar description holds for the
normalization of $\mathscr{Z}_{\overline{b}}\times_R R'$.
Therefore, by base change to $K'$, we may assume that $\nu=1$.

 Let us recall how, in the case $\nu=1$, the cohomology of $Y\times_K
 K^t$ and $Z\times_K K^t$ was computed in \cite[I.3.3]{sga7a}. For every element $d$
 of $\N'$, we set
\begin{eqnarray*}
Y_d&=&\Spec \left(\mathcal{O}(Y)[s_{d,i}\,|\,i\in
I]/((s_{d,i})^d-y_i)_{i\in I}\right),
\\[1.5ex]Z_d&=&\Spec
\left(\mathcal{O}(Z)[t_{d,i}\,|\,i\in I]/((t_{d,i})^d-z_i)_{i\in
I}\right).
\end{eqnarray*}
For $d=1$, we simply get $Y$ and $Z$. For every element $c$ of
$\N'$, we define a morphism of $Y$-schemes $Y_{cd}\to Y_d$ and a
morphism of $Z$-schemes $Z_{cd}\to Z_d$ by mapping $s_{d,i}$ to
$(s_{cd,i})^c$ and $t_{d,i}$ to $(t_{cd,i})^c$ for every $i$. All
these morphisms are finite and \'etale, because $c$ is not
divisible by $p$ and $s_{cd,i}$, $t_{cd,i}$ are units on $Y$,
resp. $Z$. In this way, we obtain projective systems of Galois
coverings of $Y$ and $Z$, and by passing to the limit, we get
procoverings $\widetilde{Y}\to Y$ and $\widetilde{Z}\to Z$.

Now the crucial point is the following. We choose a sequence of
elements $v_d$ in $\mathcal{O}(Y)$, for $d\in \N'$, such that
$v_1=v$ and such that $v_d=(v_{cd})^c$ for all $c$, $d$ in $\N'$.
This is possible because $v$ is a unit on the strictly henselian
local scheme $\mathscr{Y}_{\overline{a}}$ and the elements in
$\N'$ are not divisible by $p$. For every $d$ in $\N'$, the
morphism of $Y$-schemes
$$f_d:Y_d\to Z_d\times_Z Y:t_{d,i}\mapsto (v_d)^{\alpha_i}(s_{d,i})^{e_i}\mbox{ for all }i$$
is an isomorphism. We can construct its inverse as follows. Fix an
element $d$ in $\N'$. Since the exponents $e_i$ are either one
(for $p=0$) or powers of $p$ (for $p>0$), we know that $d$ is
prime to $e_i$ for every $i$ in $I$. We choose for every $i$ an
integer $\beta_i$ such that $d$ divides $e_i\beta_i-1$, and we
denote the quotient $(e_i\beta_i-1)/d$ by $\gamma_i$. Then the
$Y$-morphism
$$Z_d\times_Z Y\to Y_d:s_{d,i}\mapsto (v_d)^{-\alpha_i\beta_i}(y_i)^{-\gamma_i}(t_{d,i})^{\beta_i}\mbox{ for all }i$$
is inverse to $f_d$.

 The isomorphisms $f_d$
are compatible with the transition morphisms in the projective
systems $(Y_d)_{d\in \N'}$ and $(Z_d)_{d\in \N'}$, so that we
obtain by passing to the limit an isomorphism of procoverings
\begin{equation}\label{eq-procov}\widetilde{f}:\widetilde{Y}\to \widetilde{Z}\times_Z
Y.\end{equation}

The procovering $\widetilde{Z}\to Z$ factors through a morphism
$\widetilde{Z}\to Z\times_K K^t$ because we can repeatedly take
$d$-th roots of $\pi$ in $\mathcal{O}(\widetilde{Z})$ for all $d$
in $\N'$. Indeed, in $\mathcal{O}(\mathscr{Z}_{\overline{b}})$,
$\pi$ equals a unit times a monomial in the elements $z_i,\,i\in
I$, and we can always take the $d$-th root of a unit in
$\mathcal{O}(\mathscr{Z}_{\overline{b}})$ since
$\mathcal{O}(\mathscr{Z}_{\overline{b}})$ is strictly henselian
and $d$ is not divisible by $p$. Via the morphism $\widetilde{f}$
in  \eqref{eq-procov}, the $K^t$-structure on $\widetilde{Z}$
induces a $K^t$-structure on $\widetilde{Y}$ so that the
procovering $\widetilde{Y}\to Y$ factors through $\widetilde{Y}\to
Y\times_K K^t$.

By \cite[I.3.3.1]{sga7a}, the schemes $\widetilde{Y}$ and
$\widetilde{Z}$ have trivial cohomology, so that the $E_2$-terms
of the Hochschild-Serre spectral sequences associated to the
procoverings $\widetilde{Y}\to Y\times_K K^t$ and
$\widetilde{Z}\to Z\times_K K^t$ are concentrated in degrees
$(\ast,0)$ and we can use them to compute the cohomology of
$Y\times_K K^t$ and $Z\times_K K^t$ \cite[III.2.21(b)]{milne}.
 The isomorphism \eqref{eq-procov}
induces an isomorphism between these spectral sequences.
 It follows that \eqref{eq-strictloc} is an isomorphism for every
 integer $m\geq 0$.
\if false
 For every quasi-compact and quasi-separated $K$-scheme
$X$ and every integer $m\geq 0$, the natural morphism
$$\lim_{\stackrel{\longrightarrow}{d\in \N'}}H^m(X\times_K
K(d),\Lambda)\to H^m(X\times_K K^t,\Lambda)$$ is an isomorphism,
where we ordered the elements $d$ of $\N'$ by divisibility
\cite[VII.5.8]{sga4.2}. Therefore, it is enough to show that
\eqref{eq-strictloc} is an isomorphism when we replace $K^t$ by
$K(d)$, with $d$ an element of $\N'$ that is divisible by $\nu$.

Let $d$ be such an element, and set
\begin{eqnarray*}
\mathscr{Y}_{\overline{a}}(d)=\Spec
\left(\mathcal{O}(\mathscr{Y}_{\overline{a}})[s_i\,|\,i\in
I_{\overline{a}}]/(s_i^d-y_i)_{i\in I_{\overline{a}}}\right)
\\[2ex] \mathscr{Z}_{\overline{b}}(d)=\Spec
\left(\mathcal{O}(\mathscr{Z}_{\overline{b}})[t_i\,|\,i\in
I_{\overline{a}}]/(t_i^d-z_i)_{i\in I_{\overline{a}}}\right)
\end{eqnarray*}

 By \cite[I.3.3]{sga7a}, the schemes
$Z\times_K K^t$ and $Y\times_K K^t$ consist of $\nu$ connected
components that are permuted transitively by the Galois action of
$G(K^t/K)$. Since \eqref{eq-strictloc} is Galois-equivariant, it
is an isomorphism for $m=0$. \fi
\end{proof}

\begin{prop}\label{prop-tamenearby}We follow the notations
introduced in Section \ref{subsec-sncd}. Let $J$ be a non-empty
subset of $I$, and fix an integer $m\geq 0$. The restriction of
$$R^m\psi^t_{\mathscr{X}}(\Lambda)$$ to $E_J^o$ is lisse, and tamely ramified
along the irreducible components of $E_J\setminus E_J^o$.

More precisely, the sheaf
$$R^m\psi^t_{\mathscr{X}}(\Lambda)$$
becomes constant on the finite \'etale covering
$\widetilde{E}_J^o$ of $E_J^o$ of degree $N'_J\in \N'$.
\end{prop}
\begin{proof}
We may assume that $\mathscr{X}$ is connected, and we denote its
dimension by $n$. We put $q=|J|$. We may suppose that $E_J$ is
non-empty. This implies that $q\leq n$. We choose a bijection
between $J$ and $\{1,\ldots,q\}$.

 Let $x$ be a closed point
of $E_J^o$. There exist a regular system of parameters
$(x_1,\ldots,x_n)$ and a unit $u$ in $\mathcal{O}_{\mathscr{X},x}$
such that
$$\pi=u\prod_{i=1}^q (x_i)^{N_i}.$$ Let $\mathscr{U}$ be a connected affine open
neighbourhood of $x$ in $\mathscr{X}$ such that $x_1,\ldots,x_n$
and $u$ are regular functions on $\mathscr{U}$, and $u$ is a unit
in $\mathcal{O}(\mathscr{U})$.
 Consider the finite \'etale covering
$$f:\mathscr{Y}=\Spec (\mathcal{O}(\mathscr{U})[v]/(v^{N'_J}-u))\rightarrow \mathscr{U}$$
and let $y$ be a point on $\mathscr{Y}$ that is mapped to $x$ by
$f$. Since $f$ is \'etale, we have an isomorphism
$$f_s^*R^m\psi^t_{\mathscr{X}}(\Lambda)\cong
R^m\psi^t_{\mathscr{Y}}(\Lambda)$$ of constructible
$\Lambda$-sheaves on $\mathscr{Y}_s$.

The locally closed subset $E_J^o$ of $\mathscr{X}$ might not be
connected, but by the local computations in \cite[I.3.3]{sga7a},
it is enough to show that
$$R^m\psi^t_{\mathscr{X}}(\Lambda)$$ becomes constant on every
connected component of $\widetilde{E}_J^o$. By construction of the
covering $\widetilde{E}_J^o$, there is an isomorphism of
$E_J^o$-schemes
$$\widetilde{E}^o_J\times_{\mathscr{X}}\mathscr{U}\cong
\mathscr{Y}\times_{\mathscr{X}}E_J^o.$$  Thus it suffices to show
that
$$R^m\psi^t_{\mathscr{Y}}(\Lambda)$$
is constant on a Zariski-open neighbourhood of $y$ in
$\mathscr{Y}\times_{\mathscr{X}} E_J^o$. This follows from Lemma
\ref{lemm-tamenearby}, because
$$(f^*x_1,\ldots,f^*x_n)$$ is a regular system of parameters in
$\mathcal{O}_{\mathscr{Y},y}$ by Lemma \ref{lemm-para}, and
$$\pi=v^{N'_J}\prod_{i=1}^q (f^*x_i)^{N_i}$$
in $\mathcal{O}_{\mathscr{Y},y}$.
\end{proof}

\subsection{The tame monodromy zeta function}
\begin{definition}
Let $Y$ be a separated $K$-scheme of finite type. The tame
monodromy zeta function $\zeta_Y(\vart)$ of $Y$ is defined by
$$\zeta_Y(\vart)=\prod_{m\geq 0}\mathrm{det}(\vart\cdot \mathrm{Id}- \varphi\,|\,H^m_c(Y\times_K
K^t,\Q_\ell))^{(-1)^{m+1}}\ \in \Q_\ell(\vart).$$
\end{definition}
\begin{theorem}\label{thm-ac}
We follow the notations introduced in Section \ref{subsec-sncd}.
Let $Z$ be a subscheme of $\mathscr{X}_s$. We have
\begin{equation}\label{eq-AC}
\prod_{m\geq 0}\mathrm{det}(\vart\cdot \mathrm{Id}-
\varphi\,|\,\mathbb{H}^m_c(Z,R\psi^t_{\mathscr{X}}(\Q_\ell)|_Z)^{(-1)^{m+1}}=\prod_{i\in
I}(\vart^{N'_i}-1)^{-\chi(E_i^o\cap Z)}\end{equation} and, for
every element $d\in \Z_{>0}$,
\begin{equation}\label{eq-trace}\sum_{m\geq 0}(-1)^m
\mathrm{Trace}(\varphi^d\,|\,\mathbb{H}_c^m(Z,R\psi^t_{\mathscr{X}}(\Q_\ell)|_Z))=\sum_{N'_i|d}N'_i
\chi(E_i^o\cap Z).\end{equation}

In particular, if $\mathscr{X}$ is proper, then the tame monodromy
zeta function of $X=\mathscr{X}\times_R K$ is given by
\begin{equation}\label{eq-AC2}\zeta_X(\vart)=\prod_{i\in
I}(\vart^{N'_i}-1)^{-\chi(E_i^o)}\end{equation} and for every
element $d\in \Z_{>0}$, we have
\begin{equation}\label{eq-trace2}\sum_{m\geq 0}(-1)^m
\mathrm{Trace}(\varphi^d\,|\,H^m(X\times_K
K^t,\Q_\ell))=\sum_{N'_i|d}N'_i \chi(E_i^o).\end{equation}
\end{theorem}
\begin{proof}
Equations \eqref{eq-AC2} and \eqref{eq-trace2} follow from
\eqref{eq-AC} and \eqref{eq-trace}, by taking $Z=\mathscr{X}_s$
and applying the spectral sequence for tame nearby cycles
\cite[I.2.7.3]{sga7a}. For every endomorphism $M$ on a finite
dimensional vector space $V$ over a field $F$ of characteristic
zero, we have the identity \cite[1.5.3]{weil1}
$$\det(\mathrm{Id}-\vart\cdot M\,|\,V)^{-1}=\mathrm{exp}(\sum_{d>0}\mathrm{Trace}(M^d\,|\,V)\frac{\vart^d}{d})$$
in $F[[\vart]]$. Using this identity, \eqref{eq-AC} can easily be
deduced from \eqref{eq-trace} (for a similar argument, see
\cite[\S1]{A'C}). So it suffices to prove \eqref{eq-trace}. Both
sides of \eqref{eq-trace} are additive w.r.t. partitions of $Z$
into subvarieties, so that we may assume that $Z$ is contained in
$E_J^o$, for some non-empty subset $J$ of $I$, and that $Z$ is
normal. We choose a normal compactification $\overline{Z}$ of $Z$,
and a closed point $z$ on $Z$.

By the spectral sequence for hypercohomology, we have
\begin{eqnarray*}
& &\sum_{m\geq 0}(-1)^m
\mathrm{Trace}(\varphi^d\,|\,\mathbb{H}_c^m(Z,R\psi^t_{\mathscr{X}}(\Q_\ell)|_Z))\\&=&\sum_{a,b\geq
0}(-1)^{a+b}\mathrm{Trace}(\varphi^d\,|\,H^a_c(Z,R^b\psi^t_{\mathscr{X}}(\Q_\ell)|_Z))
\end{eqnarray*}
By Proposition
\ref{prop-tamenearby}, the sheaf
$$R^b\psi^t_{\mathscr{X}}(\Q_\ell)|_{Z}$$
is lisse, and tamely ramified along the irreducible components of
$\overline{Z}\setminus Z$. By the local computation in
\cite[I.3.3]{sga7a}, the action of $\varphi$ on
$$R^b\psi^t_{\mathscr{X}}(\Q_\ell)|_{Z}$$ has finite order.
 By
\cite[5.1]{NiSe} and \cite[I.3.3]{sga7a}, we have
\begin{eqnarray*}
& & \sum_{a,b\geq 0}(-1)^{a+b}
\mathrm{Trace}(\varphi^d\,|\,H^a_c(Z,R^b\psi^t_{\mathscr{X}}(\Q_\ell)|_{Z}))
\\ &=&\chi(Z)\cdot \sum_{b\geq 0}(-1)^b
 \mathrm{Trace}(\varphi^d\,|\,R^b\psi^t_{\mathscr{X}}(\Q_\ell)_z)
\\ &&
\\ &=&\left\{ \begin{array}{ll}0&\mbox{ if }|J|>1\mbox{ or }J=\{i\}\mbox{ with }N'_i\nmid
d,
\\ &
\\ N'_i\chi(Z)&\mbox{ if }J=\{i\}\mbox{ with }N'_i|d. \end{array}\right.
\end{eqnarray*}
(in \cite[5.1]{NiSe}, the condition that $Y$ is normal should be
added to the statement of the lemma). This concludes the proof.
\end{proof}
\begin{cor}\label{cor-tameeuler}
Assume that $\mathscr{X}$ is proper. If $X=\mathscr{X}\times_R K$
is cohomologically tame, then
\begin{equation*}
\chi(X)=\sum_{i\in I}N'_i\chi(E_i^o).\end{equation*} The converse
implication holds if $X$ is geometrically irreducible and of
dimension one.
\end{cor}
\begin{proof}
The opposite of the degree of $\zeta_X(\vart)$ is equal to the
tame Euler characteristic
$$\chi_{\mathrm{tame}}(X)=\sum_{m\geq 0}(-1)^m \mathrm{dim}\,H^m(X\times_K
K^t,\Q_\ell).$$ By Theorem \ref{thm-ac}, we find
$$\chi_{\mathrm{tame}}(X)=\sum_{i\in I}N'_i \chi(E_i^o).$$
 If $X$ is
cohomologically tame, then $\chi(X)=\chi_{\mathrm{tame}}(X)$. If
$X$ is geometrically irreducible and of dimension one, then the
converse implication holds as well, because the wild inertia acts
trivially on $H^m(X\times_K K^s,\Q_\ell)$ for $m\in \{0,2\}$.
\end{proof}

\begin{remark}
If $X$ is a curve, then one can use \cite[3.3]{abbes} instead of
Proposition \ref{prop-tamenearby} to prove Theorem \ref{thm-ac}.
This case suffices for the applications in Section
\ref{sec-saito}. The case of dimension $>1$ is needed in Section
\ref{sec-trace}.
\end{remark}

\section{Saito's criterion for cohomological tameness, and the
semi-stable reduction theorem}\label{sec-saito}

\subsection{Numerical criteria for cohomological tameness}
Let $C$ be a smooth, projective, geometrically connected curve of
genus $g$, and let $\mathscr{C}$ be an $sncd$-model for $C$, with
$$\mathscr{C}_s=\sum_{i\in I}N_i E_i.$$
For every non-empty subset $J$ of $I$, we define $E_J$ and $E_J^o$
as in Section \ref{subsec-sncd}. For every integer $d\geq 0$, we
denote by $I_d$ the subset of $I$ consisting of the indices $i$
such that $d|N_i$. For every
$i\in I$, we put \begin{eqnarray*} \kappa_i&=&-(E_i\cdot E_i)\\
\nu_i&=&(E_i\cdot K_{\mathscr{C}/R})\end{eqnarray*} where
$K_{\mathscr{C}/R}$ is a relative canonical divisor. We denote by
$g_i$ the genus of $E_i$. The component $E_i$ is called
\emph{principal} if $g_i>0$ or $E_i\setminus E_i^o$ contains at
least three points.

Recall the following well-known identities, for each $j\in I$:
\begin{eqnarray}
\sum_{i\in I}N_i(E_i\cdot E_j)&=&0, \label{eq1}
\\ 2g_j-2&=&\nu_j-\kappa_j, \label{eq2}
\\ 2g-2&=&\sum_{i\in I}N_i \nu_i. \label{eq3}
\end{eqnarray}
The first formula is obtained by intersecting $\mathscr{C}_s$ with
$E_j$ \cite[9.1.21]{liu}, the second and third follow from the
adjunction formula \cite[9.1.37]{liu}.

Since $\varphi$ acts trivially on $$H^m(C\times_K K^t,\Q_\ell)$$
for $m\in \{0,2\}$, it follows from \eqref{eq-AC} that the
characteristic polynomial
$$P_C(\vart)=\mathrm{det}(\vart\cdot \mathrm{Id}- \varphi\,|\,H^1(C\times_K
K^t,\Q_\ell))$$ is given by
\begin{equation}\label{eq-charpol}P_C(\vart)=(\vart-1)^2 \prod_{i\in
I}(\vart^{N'_i}-1)^{-\chi(E_i^o)}.\end{equation}

\begin{lemma}\label{lemm-euler}
We have
$$\chi(C)=\sum_{i\in I}N_i\chi(E_i^o).$$
\end{lemma}
\begin{proof}
By \eqref{eq3}, we have
$$\chi(C)=2-2g=-\sum_{i\in I}N_i \nu_i.$$
Solving $\nu_i$ from equation \eqref{eq2}, we find
$$\chi(C)=\sum_{i\in I}N_i (\chi(E_i)-\kappa_i).$$
Solving $N_i\kappa_i$ from equation \eqref{eq1}, we obtain
\begin{eqnarray*}
\chi(C)&=& \sum_{i\in I}N_i\chi(E_i)-\sum_{i\in I}\sum_{j\in
I\setminus \{i\}}N_j(E_i\cdot E_j)
\\ &=& \sum_{i\in I}N_i\chi(E_i)-\sum_{i\in I}\sum_{j\in
I\setminus \{i\}}N_i(E_i\cdot E_j)
\\ &=& \sum_{i\in I}N_i\left(\chi(E_i)-\sum_{j\in
I\setminus \{i\}}(E_i\cdot E_j)\right)
\\ &=& \sum_{i\in I}N_i\chi(E^o_i).
\end{eqnarray*}
\end{proof}

\begin{lemma}\label{lemm-ratpoint}
For each $d\in \N'$, we denote by $K(d)$ the unique extension of
$K$ in $K^t$ of degree $d$. For each couple $(i,j)$ in $I\times
I$, we consider the set $$D_{i,j}=\{\alpha_i N_i+\alpha_j
N_j\,|\,\alpha_i,\,\alpha_j\in \N\}.$$ We denote by $S$ the subset
of $I\times I$ consisting of the couples $(i,j)$ such that
$E^o_{\{i,j\}}\neq \emptyset$, and we put
$$D_{\mathscr{C}}=\N'\cap (\cup_{(i,j)\in S} D_{i,j}).$$
Then for each $d\in \N'$, the set $C(K(d))$ is non-empty  if and
only if $d$ belongs to $D_{\mathscr{C}}$.
\end{lemma}
\begin{proof}
Assume that $C(K(d))$ is non-empty. Let $a$ be an element of
$C(K(d))$, and denote by $R(d)$ the normalization of $R$ in
$K(d)$. By the valuative criterion for properness, the point $a$
extends uniquely to a section $\psi_a$ in $\mathscr{C}(R(d))$. We
denote by $a_0$ the image in $\mathscr{C}_s$ of the closed point
of $\Spec R(d)$. The point $a_0$ belongs to $E_{\{i,j\}}^o$, for
some couple $(i,j)$ in $S$, and this couple is unique up to
transposition. There exist elements $x,\,y,\,u$ in
$\mathcal{O}_{\mathscr{C},a_0}$ such that $u$ is a unit and $\pi=u
 x^{N_i}y^{N_j}$. If we denote by $v_{K(d)}$ the normalized
discrete valuation on $K(d)$, then the equality
$$\pi=\psi_a^*(u)\psi_a^*(x)^{N_i}\psi_a^*(y)^{N_j}$$ in $R(d)$ implies that
$$d=N_i\cdot v_{K(d)}(\psi_a^*(x))+N_j\cdot v_{K(d)}(\psi_a^*(y)).$$ It follows that  $d\in D_{i,j}$.


So let us show the converse implication. Let $d$ be an element of
$\N'\cap D_{i,j}$, for some $(i,j)\in S$. We'll treat the case
where $i\neq j$, the other case can be proven in a similar
fashion. We may assume that $d\notin D_{i,i}$ and $d\notin
D_{j,j}$. Then there exist elements $\alpha_i$, $\alpha_j$ in
$\N_0$ such that $\alpha_i N_i+\alpha_j N_j=d$. We set
$N=\gcd(N_i,N_j)$.
 By B\'ezout's theorem, we can find integers $m_i$ and $m_j$
such that $N=m_i N_i+m_j N_j$. Note that $N$ divides $d$, so that
$N$ must belong to $\N'$.

Let $b$ be a point of $E_{\{i,j\}}^o$. There exist a regular
system of parameters $(x,y)$ and a unit $u$ in
$\mathcal{O}_{\mathscr{C},b}$ such that
$$\pi=u  x^{N_i}y^{N_j}.$$ We take an integral affine \'etale
neighbourhood $U=\Spec B$ of $b$ in $\mathscr{C}$ such that
$u,\,x,\,y$ are regular functions on $U$ and $u=v^{N}$ for some
unit $v$ in $B$. Then we can define a morphism
$$f:U\rightarrow V=\Spec R[x',y']/(\pi-(x')^{N_i}(y')^{N_j})$$
by $x'\mapsto v^{m_i}x$ and $y'\mapsto v^{m_j}y$. This morphism is
\'etale at every point of $U\times_{\mathscr{C}}b$, by Lemma
\ref{lemm-para}. Since $R(d)$ is strictly henselian, it suffices
to show that $V(R(d))$ contains a section that maps the closed
point of $\Spec R(d)$ to the origin in $V_s$; this section will
then lift to $U$. We can construct such a section by sending $x'$
to $\pi(d)^{\alpha_i}$ and $y'$ to $\pi(d)^{\alpha_j}$, where
$\pi(d)$ is an element of $R(d)$ such that $\pi(d)^d=\pi$. This
concludes the proof.
\end{proof}
\begin{cor}\label{cor-ratpoint}
The set $C(K)$ is non-empty if and only if there exists an element
$\alpha$ in $I$ such that $N_\alpha=1$. The set $C(K^t)$ is
non-empty if and only if there exists an element $\beta$ in $I$
with $p\nmid N_\beta$.
\end{cor}
We will repeatedly use the following elementary lemma.

\begin{lemma}\label{lemma-tree}
Let $I'$ be a non-empty subset of $I$ such that $\cup_{i\in
I'}E_i$ is connected. For each $i\in I'$, we put
$$\overline{E}_i^o=E_i\setminus \cup_{j\in I'\setminus \{i\}}E_j.$$
Then $\sum_{i\in I'}\chi(\overline{E}_i^o)\leq 0$, unless
$\cup_{i\in I'}E_i$ is a tree of rational curves. In the latter
case, $\sum_{i\in I'}\chi(\overline{E}_i^o)=2$.
\end{lemma}
\begin{proof}
This is easily proven by induction on $|J|$; see
\cite[2.2]{rodrigues}.
\end{proof}


\begin{lemma}\label{lemm-pos}
Fix an integer $d>1$, and let $I'$ be a subset of $I_d$ such that
$\cup_{i\in I'}E_i$ is a connected component of $\cup_{i\in
I_d}E_i$. Then we have
$$\sum_{i\in I'}\chi(E_i^o)\leq 0.$$
\end{lemma}
\begin{proof}
First, suppose that $I'=I$, and that
$$\sum_{i\in I'}\chi(E_i^o)> 0.$$ Then $\mathscr{C}_s$ is a tree of rational
curves, by Lemma \ref{lemma-tree}. By Lemma \ref{lemm-euler}, we
find that $\chi(C)$ is at least $2d>2$, which is impossible.

Hence, we may assume that there exists an index $a\in I'$ such
that $E_a$ meets a component of $\mathscr{C}_s$ whose multiplicity
is not divisible by $d$. By \eqref{eq1}, the intersection of $E_a$
with
$$\cup_{i\in I\setminus I_d}E_i$$ contains at least two points. For each $i\in I'$, we put
$$\overline{E}_i^o=E_i\setminus (\cup_{j\in I'\setminus
\{i\}}E_j).$$ Then
$$\sum_{i\in I'}\chi(\overline{E}^o_i)\leq 2$$ by
Lemma \ref{lemma-tree}, and since $\cup_{i\in I'}E_i^o$ is an open
subset of $\cup_{i\in I'}\overline{E}_i^o$ whose complement
contains at least two points, we find
$$\sum_{i\in I'}\chi(E_i^o)\leq 0.$$
\end{proof}
\begin{prop}\label{prop-poly}
The rational function
$$Q_C(\vart)=(\vart-1)^2\prod_{i\in I}(\vart^{N_i}-1)^{-\chi(E_i^o)}$$
is a polynomial in $\Z[\vart]$. It is divisible by the
characteristic polynomial $P_{C}(\vart)$ of the tame monodromy
operator $\varphi$ on $H^1(C\times_K K^t,\Q_\ell)$.
\end{prop}
\begin{proof} By \eqref{eq-charpol}, we have $P_C(t)=Q_C(t)$ if
$p=0$, so that we may assume that $p>0$. The prime factorization
of $Q_C(\vart)$ is given by
$$Q_C(\vart)=(\vart-1)^{2-\sum_{i\in I}\chi(E_i^o)}\prod_{d>1}\Phi_d(\vart)^{-\sum_{i\in I_d}\chi(E_i^o)}.$$
It follows from Lemma \ref{lemma-tree} that
$$2-\sum_{i\in I}\chi(E_i^o)\geq 0.$$
By Lemma \ref{lemm-pos}, we know that that $$-\sum_{i\in
I_d}\chi(E_i^o)\geq 0$$ for all $d>1$. It follows that
$Q_C(\vart)$ is a polynomial. By formula \eqref{eq-charpol}, we
have
$$\frac{Q_C(\vart)}{P_C(\vart)}=\prod_{d>0}\Phi_{dp}(\vart)^{-\sum_{i\in
I_{dp}}\chi(E_i^o)}$$ which is a polynomial by Lemma
\ref{lemm-pos}.
\end{proof}
%
\begin{cor}\label{cor-tamecrit3}
The following are equivalent: \begin{enumerate} \item the curve
$C$ is cohomologically tame, \item we have
$$\sum_{i\in I}N_i\chi(E_i^o)=\sum_{i\in I}N'_i\chi(E_i^o),$$
\item we have
$$P_C(\vart)=(\vart-1)^2\prod_{i\in I}(\vart^{N'_i}-1)^{-\chi(E_i^o)}=(\vart-1)^2\prod_{i\in I}(\vart^{N_i}-1)^{-\chi(E_i^o)}.$$
\end{enumerate}
\end{cor}
\begin{proof}
 The
equivalence of (1) and (2) follows from Corollary
\ref{cor-tameeuler} and Lemma \ref{lemm-euler}. Point (3) implies
(2) by comparing degrees. If (2) holds, then $P_C(\vart)$ and
$Q_C(\vart)$ are monic polynomials of the same degree, so they
coincide because $P_C(\vart)$ divides $Q_C(\vart)$ by Proposition
\ref{prop-poly}. Hence, (2) implies (3).
\end{proof}

\subsection{Tame models}
\begin{definition}
Let $d$ be an element of $\N$. We say that $\mathscr{C}$ is
$d$-tame if $I\neq I_d$ and, for each $i\in I_d$, we have
$\chi(E_i^o)=0$.
\end{definition}
In particular, $\mathscr{C}$ is always $0$-tame.
\begin{lemma}\label{lemm-dtame}
Let $d$ be an element of $\Z_{>1}$.
 Then $\mathscr{C}$ is $d$-tame iff for each $i\in I_d$, the
following properties hold:
\begin{itemize}
\item $E_i\cong \mathbb{P}^1_k$, \item $E_i\setminus E_i^o$
consists of precisely two points, \item if $E_i$ intersects $E_j$
with $j\in I\setminus \{i\}$, then $N_j$ is not divisible by $d$.
\end{itemize}
\end{lemma}
\begin{proof}
The ``if'' part is trivial, so let us prove the converse
implication. Assume that $\mathscr{C}$ is $d$-tame, and let
$\alpha$ be an element of $I$ such that $d|N_\alpha$. We know that
$\chi(E_\alpha^o)=0$. So either $E_\alpha=E_\alpha^o$ and
$E_\alpha$ is an elliptic curve, or $E_\alpha$ is a rational curve
and $E_\alpha\setminus E_\alpha^o$ consists of precisely two
points. The first possibility cannot occur, since it would imply
that $I=I_d=\{\alpha\}$.

Assume that $E_\alpha$ meets precisely one other component
$E_\beta$ of $\mathscr{C}_s$. Then $E_\alpha\cap E_\beta$ consists
of exactly two points. Assume that $d|N_\beta$. Since
$\mathscr{C}$ is $d$-tame, we see that $E_\beta$ is rational, and
that $E_\beta$ meets no other components of $\mathscr{C}_s$.
Hence, $I=\{\alpha,\beta\}$. This contradicts the fact that $I\neq
I_d$.

So we may assume that $E_\alpha$ meets precisely two other
components $E_\beta$ and $E_\gamma$ of $\mathscr{C}_s$, each of
them in exactly one point. It suffices to show that $d\nmid
N_\beta$ and $d\nmid N_\gamma$. If $d|N_\beta$, then $d|N_\gamma$
by equation \eqref{eq1} (applied to $j=\alpha$). Repeating the
arguments (with $\alpha$ replaced by $\beta$, resp. $\gamma$), we
find that $\mathscr{C}_s$ is a loop of rational curves, and that
$d|N_i$ for each $i\in I$. This contradicts the $d$-tameness of
$\mathscr{C}$.
\end{proof}

\begin{theorem}\label{thm-primed}
We fix an integer $d>1$. Assume that $\mathscr{C}$ is a relatively
minimal $sncd$-model of $C$. If $I\neq I_d$, then the following
are equivalent:
\begin{enumerate}
\item the polynomial
$$Q_C(\vart)=(\vart-1)^2\prod_{i\in I}(\vart^{N_i}-1)^{-\chi(E_i^o)}\in \Z[\vart]$$
has no root whose order in $\mathbb{G}_m(\Q^a)$ is divisible by
$d$, \item $\mathscr{C}$ is $d$-tame.
\end{enumerate}
Moreover, if $I=I_d$ and (1) holds, then $g=1$.
\end{theorem}
\begin{proof}
It is obvious that (2) implies (1). Assume, conversely, that (1)
holds. For each integer $m>0$, we denote by $I_{=m}$ the subset of
$I$ consisting of the indices $i$ with $N_i=m$. By our assumption,
we have
\begin{equation}\label{eq-divdn}\sum_{i\in
I_{dn}}\chi(E_i^o)=0\end{equation} for each $n\in \Z_{>0}$, since
this is the exponent of the cyclotomic polynomial
$\Phi_{dn}(\vart)$ in the prime factorization of $Q_C(\vart)$.
Taking linear combinations of these equations, we see that
\begin{equation}\label{eq-isdn} \sum_{i\in
I_{=dn}}\chi(E_i^o)=0\end{equation} for each $n\in \Z_{>0}$.
 In particular, if $I=I_d$, Lemma \ref{lemm-euler} implies that
$g=1$. Hence, we may assume that $I\neq I_d$.

Suppose that $\mathscr{C}$ is not $d$-tame. Then there exists an
index $\alpha\in I$ such that $d|N_\alpha$ and
$\chi(E_\alpha^o)\neq 0$. We choose such an $\alpha$ with maximal
$N_\alpha$. By \eqref{eq-isdn}, we may assume that
$\chi(E_\alpha^o)> 0$. Then $E_\alpha$ is rational, and $E_\alpha$
meets exactly one other component $E_{\beta}$ of $\mathscr{C}_s$,
in precisely one point. By \eqref{eq1}, we know that
$N_{\beta}=\kappa_\alpha N_\alpha$. Since $\mathscr{C}$ is
relatively minimal, $\kappa_\alpha$ must be at least $2$, by
Castelnuevo's criterion and Proposition \ref{prop-liucor}. It
follows that $\beta$ belongs to $I_d$, and that $N_{\beta}\geq
2N_\alpha$. By maximality of $N_{\alpha}$, we must have
$\chi(E_{\beta}^o)= 0$.
 Hence,
$E_{\beta}$ is rational and meets precisely one component
$E_{\gamma}$ of $\mathscr{C}_s$ distinct from $E_{\alpha}$, in
exactly one point. Again applying \eqref{eq1}, Castelnuevo's
criterion and Proposition \ref{prop-liucor}, we find that
$$\kappa_{\beta}N_\beta=N_\alpha+N_\gamma$$
and $\kappa_\beta\geq 2$, so that $d|N_\gamma$ and
$N_\gamma>N_\beta$. Repeating the arguments, we produce an
infinite chain of rational curves in $\mathscr{C}_s$, which is
impossible.
\end{proof}
\begin{cor}\label{cor-primed}
Assume that $\mathscr{C}$ is a relatively minimal $sncd$-model of
$C$, and that $C$ is cohomologically tame. Suppose either that
$g\neq 1$ or that $C$ is an elliptic curve. Then for each integer
$d>1$, the following are equivalent: \begin{enumerate}\item the
model $\mathscr{C}$ is $d$-tame, \item $\varphi$ has no eigenvalue
on $H^1(C\times_K K^t,\Q_\ell)$ whose order in
$\mathbb{G}_m(\Q_\ell^a)$ is divisible by $d$.
\end{enumerate}
In particular, $\mathscr{C}$ is $p$-tame.
\end{cor}
\begin{proof}
By Corollary \ref{cor-tamecrit3}, $Q_C(\vart)=P_C(\vart)$.
Clearly, property (1) cannot hold if $I=I_d$. Neither can (2): by
Theorem \ref{thm-primed}, the conjunction of (2) and $I=I_d$ would
imply that $g=1$, so that $C$ would have a rational point, by our
assumptions. This contradicts $I=I_d$, by Corollary
\ref{cor-ratpoint}. Hence, we may assume that $I\neq I_d$. From
Theorem \ref{thm-primed}, we get the equivalence of (1) and (2).
Then $p$-tameness follows from the fact that the pro-$p$-part of
$G(K^t/K)$ is trivial, so that the order of an eigenvalue of
$\varphi$ in $\mathbb{G}_m(\Q_\ell^a)$ cannot be divisible by $p$.
\end{proof}

\subsection{Saito's criterion for cohomological tameness}

\begin{definition}
We denote by $Jac(C)$ the Jacobian of $C$. We say that $C$ is
pseudo-wild if $g=1$,  $C(K^t)$ is empty, and one of the following
holds:
\begin{itemize}
\item $p>3$, \item $p=2$, and $Jac(C)$ has reduction type $I_n\
(n\geq 0)$ , $IV$ or $IV^*$, \item $p=3$ and $Jac(C)$ has
reduction type $I_n$, $I^*_n$ $(n\geq 0)$, $III$, or $III^*$.
\end{itemize}
\end{definition}

\begin{theorem}[Saito's criterion for cohomological
tameness 
]\label{thm-saito} Assume that $\mathscr{C}$
is a relatively minimal $sncd$-model of $C$. The following are
equivalent:
\begin{enumerate}
\item the curve $C$ is cohomologically tame,
\item one of the following two conditions is
satisfied:\begin{itemize} \item $\mathscr{C}$ is $p$-tame, \item
$C$ is pseudo-wild.
\end{itemize}
\end{enumerate}
In (2), the two conditions are disjoint, i.e., at most one of them
can hold.
\end{theorem}
\begin{proof}
Let us first explain why the conditions in (2) are disjoint.
Assume that $C$ is pseudo-wild. By Corollary \ref{cor-ratpoint},
emptiness of $C(K^t)$ implies that $p|N_i$ for all $i\in I$.
Hence, $\mathscr{C}$ cannot be $p$-tame.

It follows immediately from Corollary \ref{cor-tamecrit3} that (2)
implies (1). It remains to show that (1) implies (2). By Corollary
\ref{cor-primed}, it suffices to consider the case where $g=1$ and
$C(K)$ is empty. We may also assume that $I=I_p$, since otherwise,
Theorem \ref{thm-primed} implies that $\mathscr{C}$ is $p$-tame.
%

 The equality $I=I_p$ implies that $C(K^t)$ is empty, by Corollary \ref{cor-ratpoint}. Since $C$ is cohomologically tame, the same holds for
its Jacobian $Jac(C)$. Looking at the Kodaira-N\'eron reduction
table, and applying Corollary \ref{cor-primed}, we see that an
elliptic $K$-curve $E$
 is cohomologically tame iff one of the following conditions
 holds:
 \begin{itemize}
\item $p>3$, \item $p=2$ and $E$ has reduction type $I_n\ (n\geq
0)$ , $IV$ or $IV^*$, \item $p=3$ and $E$ has reduction type
$I_n$, $I^*_n$ $(n\geq 0)$, $III$, or $III^*$.
\end{itemize}
Applying this criterion to the case $E=Jac(C)$, we find that $C$
is pseudo-wild.
\end{proof}
\begin{cor}
The following are equivalent:
\begin{enumerate}
\item $C$ is pseudo-wild, \item $C$ is cohomologically tame, and
$C(K^t)$ is empty.
\end{enumerate}
\end{cor}
\begin{proof}
The implication $(1)\Rightarrow (2)$ follows from Theorem
\ref{thm-saito}. Conversely, assume that $(2)$ holds, and suppose
that $\mathscr{C}$ is a relatively minimal $sncd$-model of $C$.
Emptiness of $C(K^t)$ implies $I=I_p$, by Corollary
\ref{cor-ratpoint}, so that $\mathscr{C}$ is not $p$-tame. Hence,
by Theorem \ref{thm-saito}, $C$ is pseudo-wild.
\end{proof}

\subsection{The semi-stable reduction theorem}
Recall the following well-known lemma.
\begin{lemma}\label{lemm-finquot}
If the wild inertia $P$ acts continuously on a finite dimensional
$\Q_\ell$-vector space $V$, then the action factors through a
finite quotient of $P$.
\end{lemma}
\begin{proof}
 Continuity of the action implies that $V$ admits a
 $\Z_\ell$-lattice $M$ that is stable under the action of $P$ \cite[\S1.1]{serre-repres}.
The image $P_0$ of $P$ in the automorphism group $\Aut(M)$ is a
pro-$p$-group, so that is has trivial intersection with the
pro-$\ell$-group
$$\ker(\Aut(M)\rightarrow \Aut(M/\ell M)).$$ It follows that $P_0$
is isomorphic to a subgroup of the finite group $\Aut(M/\ell M)$.
\end{proof}

\begin{theorem}[Cohomological criterion for semi-stable
reduction; Deligne - Mumford, Saito]\label{thm-sstable} Assume
that $\mathscr{D}$ is a relatively minimal $ncd$-model of $C$. We
assume that either $g\neq 1$, or $C$ is an elliptic curve. Then
$\mathscr{D}$ is semi-stable iff the monodromy action on
$$H^1(C\times_K K^s,\Q_\ell)$$ is unipotent.
\end{theorem}
\begin{proof}
Let $\mathscr{C}$ be the relatively minimal $sncd$-model of $C$
obtained by blowing up the self-intersection points of the
irreducible components of $\mathscr{D}_s$. We write
$$\mathscr{C}_s=\sum_{i\in I}N_i E_i$$
as before.

If $\mathscr{D}$ is semi-stable, then $C$ is cohomologically tame,
by Corollary \ref{cor-tamecrit3}. The action of $\varphi$ on
$$H^1(C\times_K K^t,\Q_\ell)$$ is unipotent, by formula
\eqref{eq-charpol}. It follows that the monodromy action on
$$H^1(C\times_K K^s,\Q_\ell)$$ is unipotent.

Conversely, suppose that the monodromy action on $$H^1(C\times_K
K^s,\Q_\ell)$$ is unipotent. Then $C$ is cohomologically tame, by
Lemma \ref{lemm-finquot}. We denote by $I_{>1}$ the subset of $I$
consisting of indices $i$ with $N_i>1$. By Corollary
\ref{cor-primed}, we know that $\mathscr{C}$ is $d$-tame for all
$d>1$. In particular, $\chi(E_i^o)=0$ for each $i\in I_{>1}$.
Hence, by Lemma \ref{lemm-euler}, we must have $I\neq I_{>1}$ if
$g\neq 1$. If $g=1$, then we also have $I\neq I_{>1}$ by Corollary
\ref{cor-ratpoint}, since $C$ has a rational point.


Let $j$ be an element of $I_{>1}$.
 Then by Lemma \ref{lemm-dtame}, $E_{j}$ is
rational, and
 $E_{j}\setminus E_{j}^o$ consists of exactly two points.
 Since $I\neq I_{>1}$, the component $E_{j}$ fits into a
 sequence of components
$$E_{j_0},\ldots,E_{j_{a+1}}$$
satisfying the conditions of \cite[5.1]{halle}, so that  the
component $E_{j}$ is contracted by $h$. Hence, $\mathscr{D}$ is
semi-stable.
\end{proof}
\begin{cor}[Semi-stable reduction theorem; Deligne-Mumford]\label{cor-sstable}
There exists a finite separable extension $K'$ of $K$ such that
every relatively minimal $ncd$-model of $C\times_K K'$ is
semi-stable.
\end{cor}
\begin{proof}
The monodromy action on
$$H^1(C\times_K K^s,\Q_\ell)$$
is quasi-unipotent, by the monodromy theorem \cite[I.1.3]{sga7a}.
 This also follows immediately from \eqref{eq-charpol} and Lemma
  \ref{lemm-finquot}.
\end{proof}

\begin{cor}\label{cor-mindegree}
Assume that $\mathscr{C}$ is a relatively minimal $sncd$-model of
$C$, and that $C$ is cohomologically tame. Suppose either that
$g\neq 1$ or that $C$ is an elliptic curve. The degree $e$ of the
minimal extension of $K$ where $C$ acquires semi-stable reduction
is equal to
$$\mathrm{lcm}\,\{N_i\,|\,i\in I,\,E_i\ \mathrm{is\ principal}\}.$$
\end{cor}
\begin{proof}
Note that $E_i$ is principal iff $\chi(E_i^o)<0$, or $E_i=E_i^o$
and $E_i$ is an elliptic curve. The latter possibility only occurs
if $C$ is an elliptic curve with good reduction, in which case the
statement is obvious.

It follows from Corollary \ref{cor-primed} that
\begin{eqnarray}e&=&\mathrm{lcm}\,\{d\in \Z_{>1}\,|\,\mathscr{C}\mbox{ is not }d\mbox{-tame}\}
\\ &=& \mathrm{lcm}\,\{N_i\,|\,\chi(E^o_i)\neq 0\mbox{ or
}I=I_{N_i}\}.  \label{eq-mindeg}
\end{eqnarray} On the other hand, by the implication
(1)$\Rightarrow(3)$ in Corollary \ref{cor-tamecrit3}, we know that
$e$ divides
$$\mathrm{lcm}\,\{N_i\,|\,\chi(E_i^o)<0\}.$$
This value divides the right hand side of \eqref{eq-mindeg}. It
follows that
$$e=\mathrm{lcm}\,\{N_i\,|\,\chi(E_i^o)<0\}.$$
\end{proof}
If $g>1$, Corollary \ref{cor-mindegree} was proven by a different
method in \cite[7.5]{halle}.


\begin{remark}
If $C$ has genus at least one, then it is not hard to show that
the minimal $ncd$-model $\mathscr{D}$ of $C$ is semi-stable if and
only if the minimal regular model $\mathscr{E}$ of $C$ is a
semi-stable $ncd$-model. The ``if'' part is obvious. The model
$\mathscr{D}$ is obtained by blowing up $\mathscr{E}$ at points of
$\mathscr{E}_s$ where $\mathscr{E}_s$ does not have normal
crossings. Such a point is never a regular point of
$\mathscr{E}_s$, so that the exceptional divisor of the blow-up
has multiplicity at least two in $\mathscr{D}_s$. This means that
$\mathscr{D}\to \mathscr{E}$ must be an isomorphism if
$\mathscr{D}$ is semi-stable.
\end{remark}

\subsection{Some counterexamples}
To conclude this section, we discuss some examples that show that
certain conditions in the statements of the above results cannot
be omitted.
\begin{enumerate}
\item Theorem \ref{thm-sstable} is false if we take for
$\mathscr{D}$ a relatively minimal $sncd$-model of $C$. If $C$ is
an elliptic curve of type $I_1$, then its minimal $ncd$-model
$\mathscr{C}$ is semi-stable, but the special fiber of its minimal
$sncd$-model contains a component of multiplicty $2$ (the
exceptional curve of the blow-up of $\mathscr{C}$ at the
self-intersection point of $\mathscr{C}_s$).

\item Theorem \ref{thm-sstable} can fail for curves of genus one
without rational point, namely, for non-trivial $E$-torsors over
$K$, with $E$ an elliptic $K$-curve with semi-stable reduction.

 \item Theorem \ref{thm-saito} is false if we don't
assume that $\mathscr{C}$ is relatively minimal. If $N_i$ is
divisible by $p$, then blowing up a point of $E_i^o$ destroys
$p$-tameness of the model.

\item Corollaries \ref{cor-primed} and \ref{cor-mindegree} can
fail for genus one curves without rational point, for instance,
for non-trivial $E$-torsors over $K$, with $E$ an elliptic
$K$-curve with good reduction.

\item Corollary \ref{cor-mindegree} is false if we do not assume
that $C$ is cohomologically tame. For instance, if $k$ has
characteristic $2$ and $R$ is the ring of Witt vectors over $k$,
then the elliptic $K$-curve with Weierstrass equation $y^2=x^3+2$
has reduction type $II$, while it acquires good reduction over
$K(\sqrt{2})$.
\end{enumerate}

\section{The trace formula}\label{sec-trace}

\subsection{The rational volume and the trace formula}
Let $X$ be a smooth and proper $K$-variety. Recall that a weak
N\'eron model for $X$ is a separated smooth $R$-scheme of finite
type $\X$, endowed with an isomorphism $\X\times_R K\cong X$, such
that the natural map
$$\X(R)\rightarrow \X(K)=X(K)$$ is a bijection. Such a
weak N\'eron model always exists, and it can be constructed by
taking a N\'eron smoothening of a proper $R$-model of $X$
\cite[3.1.3]{neron}.

It follows from
  \cite[4.5.3]{motrigid} and
\cite[5.2]{Ni-tracevar} (see also \cite[5.4]{nise-err} for an
erratum) that the value
$$s(X)=\chi(\X_s)\quad \in \Z$$
only depends on $X$, and not on the choice of weak N\'eron model.

\begin{definition}
We call $s(X)$ the rational volume of $X$.
\end{definition}
\begin{remark}
It is quite non-trivial that $s(X)$ is independent of the choice
of weak N\'eron model. If $k$ has positive characteristic, we do
not know any proof of this result that does not use the change of
variables formula for motivic integrals. If $k$ has characteristic
zero, it can be deduced from the trace formula (Corollary
\ref{cor-notwild}).
\end{remark}

The value $s(X)$ is a measure for the set of rational points on
$X$. In particular, $s(X)=0$ if $X(K)=\emptyset$, since in this
case, $X$ is a weak N\'eron model of itself. In
\cite{Ni-tracevar}, we've shown that under a certain tameness
condition on $X$, the value $s(X)$ admits a cohomological
interpretation in terms of a trace formula. To study this formula,
we introduce the following definition.
\begin{definition}
We define the error term $\varepsilon(X)$ by
$$\varepsilon(X)=\sum_{m\geq 0}(-1)^m \mathrm{Trace}(\varphi\,|\,H^m(X\times_K K^t,\Q_\ell))-s(X).$$
We say that the trace formula holds for $X$ if $\varepsilon(X)=0$,
i.e., if
$$s(X)=\sum_{m\geq 0}(-1)^m \mathrm{Trace}(\varphi\,|\,H^m(X\times_K
K^t,\Q_\ell)).$$
\end{definition}
In \cite[\S1]{Ni-abelian}, we raised the following question.
\begin{question}\label{ques-traceform}
Let $X$ be a smooth, proper, geometrically connected $K$-variety.
Assume that $X$ is cohomologically tame and that $X(K^t)$ is
non-empty. Is it true that the trace formula holds for $X$?
\end{question}
We've proven that this question has a positive answer if $k$ has
characteristic zero \cite[6.5]{Ni-tracevar}, if $X$ is a curve
\cite[\S7]{Ni-tracevar} and also if $X$ is an abelian variety
\cite[2.9]{Ni-abelian}. The condition that $X(K^t)$ is non-empty
can not be omitted, by \cite[7.7]{Ni-tracevar}. If $X$ is not
cohomologically tame, it would be quite interesting to relate
$\varepsilon(X)$ to other measures of wild ramification. However,
by the example in \cite[7.7]{Ni-tracevar}, the value
$\varepsilon(X)$ can not always be computed from the Chow motive
of $X$, if we don't impose the condition $X(K^t)\neq \emptyset$.

\subsection{Computation of the error term}
In \cite[7.3]{Ni-tracevar}, we gave an explicit formula for the
error term $\varepsilon(X)$ in terms of an $sncd$-model for $X$,
if $X$ is a curve. Thanks to Theorem \ref{thm-ac}, we can
generalize this result to arbitrary dimension.
\begin{theorem}\label{thm-error}
Let $X$ be a smooth and proper $K$-variety, and assume that $X$
admits an $sncd$-model $\mathscr{X}$. We denote by $\{E_i\,|\,i\in
I\}$ the set of irreducible components of $\mathscr{X}_s$, and we
write
$$\mathscr{X}_s=\sum_{i\in I}N_i E_i.$$
 We define the subset $I^w$ of $I$ by $$I^w=\{i\in
I\,|\,N_i=p^a\mbox{ for some }a\in \Z_{>0}\}.$$ Then the error
term $\varepsilon(X)$ is given by
$$\varepsilon(X)=\sum_{i\in I^w}\chi(E_i^o).$$
\end{theorem}
Note that the set $I^w$ is empty if $p=0$.
\begin{proof}
Since $\mathscr{X}$ is a regular proper $R$-model of $X$, its
$R$-smooth locus $Sm(\mathscr{X})$ is a weak N\'eron model for $X$
(see the remark following \cite[3.1.2]{neron}). Therefore,
$$s(X)=\sum_{N_i=1}\chi(E_i^o).$$
On the other hand, by Theorem \ref{thm-ac}, we have
$$\sum_{m\geq
0}(-1)^m \mathrm{Trace}(\varphi\,|\,H^m(X\times_K
K^t,\Q_\ell))=\sum_{N'_i=1}\chi(E_i^o).$$ It follows that
$$\varepsilon(X)=\sum_{i\in I^w}\chi(E_i^o).$$
\end{proof}
\begin{cor}
The value
$$\sum_{i\in I^w}\chi(E_i^o)$$
only depends on $X$, and not on the $sncd$-model $\mathscr{X}$.
\end{cor}
\begin{cor}\label{cor-notwild}
If $I^w=\emptyset$, then the trace formula holds for $X$. In
particular, if $p=0$, then the trace formula holds for every
smooth and proper $K$-variety.
\end{cor}

By Theorem \ref{thm-error}, Question \ref{ques-traceform} implies
the following one.

\begin{question}\label{ques-saito} Let $X$ be a
smooth, proper, geometrically connected $K$-variety. Assume that
$X$ is cohomologically tame, that $X(K^t)$ is non-empty, and that
$X$ admits an $sncd$-model $\mathscr{X}$, with
$\mathscr{X}_s=\sum_{i\in I}N_i E_i$. Is it true that
$$\sum_{i\in I^w}\chi(E_i^o)=0?$$
\end{question}

A positive answer to this question would form a partial
generalization of Saito's criterion for cohomological tameness
(Theorem \ref{thm-saito}) to arbitrary dimension.

\end{document}